\documentclass{amsart}
\usepackage{amssymb}
\usepackage{float}
\usepackage[pdftex]{hyperref}
\newtheorem{thm}{Theorem}[section]
\newtheorem{lem}[thm]{Lemma}
\newtheorem{prop}[thm]{Proposition}
\newtheorem{cor}[thm]{Corollary}

\theoremstyle{definition}
\newtheorem{defn}[thm]{Definition}

\theoremstyle{remark}

\newcommand{\system}[1]{\mbox{\fontfamily{cmss}\fontshape{n}\fontseries{m}%
    \selectfont#1}}
\newcommand{\RCA}{\system{RCA}\ensuremath{_0}}
\newcommand{\WKL}{\system{WKL}\ensuremath{_0}}
\newcommand{\WWKL}{\system{WWKL}\ensuremath{_0}}

\newcommand{\ACA}{\system{ACA}\ensuremath{_0}}

\newcommand{\REC}{\system{REC}}

\newcommand{\DNR}{\system{DNR}\ensuremath{_0}}
\newcommand{\Gepsilon}{\system{G}\ensuremath{_{\delta}\system{-}\epsilon}}
\newcommand{\Positive}{\system{POS}}
\newcommand{\Gdelta}{\system{G}\ensuremath{_{\delta}}\system{-REG}}

\renewcommand{\epsilon}{\varepsilon}
\renewcommand{\phi}{\varphi}
\renewcommand{\models}{\vDash}
\newcommand{\notmodels}{\nvDash}

\newcommand{\dom}{\qopname\relax o{dom}}
\newcommand{\seq}[1]{{\langle{#1}\rangle}}

\newcommand{\Q}{\mathbb{Q}}
\newcommand{\N}{\mathbb{N}}

\newcommand{\LR}{{\emph{LR}}}
\newcommand{\LK}{{\emph{LK}}}
\newcommand{\cat}{\widehat{\;\;\:}}

\begin{document}

\title{Lowness notions, measure and domination}

\author{Bj{\o}rn Kjos-Hanssen}
\address{Bj{\o}rn Kjos-Hanssen, Department of Mathematics,
University of Hawai{\textquoteleft}i at M{\=a}noa, 
Honolulu, HI 96822}
\email{bjoern@math.hawaii.edu}

\author{Joseph S.~Miller}
\address{Joseph S.~Miller, Department of Mathematics,
University of Wisconsin, Madison, WI 53706-1388}
\email{jmiller@math.wisc.edu}

\author{Reed Solomon}
\address{Reed Solomon, Department of Mathematics,
University of Connecticut, Storrs, CT 06269-3009}
\email{solomon@math.uconn.edu}

\thanks{Solomon's research was partially funded by NSF Grant DMS-0400754. Miller's was supported by NSF grants DMS-0945187 and DMS-0946325, the latter being part of a Focused Research Group in Algorithmic Randomness. Kjos-Hanssen was supported by NSF Grants DMS-0901020 and DMS-0652669 (the latter part of the FRG in Algorithmic Randomness).}

\makeatletter
\@namedef{subjclassname@2010}{\textup{2010} Mathematics Subject Classification}
\makeatother

\begin{abstract}
We show that positive measure domination implies uniform almost everywhere domination and that this proof translates into a proof in the subsystem $\WWKL$ (but not in $\RCA$) of the equivalence of various Lebesgue measure regularity statements introduced by Dobrinen and Simpson. This work also allows us to prove that low for weak $2$-randomness is the same as low for Martin-L\"of randomness (a result independently obtained by Nies). Using the same technique, we show that $\leq_\LR$ implies $\leq_\LK$, generalizing the fact that low for Martin-L\"of randomness implies low for $K$.
\end{abstract}

\maketitle

\section{Introduction}
\label{sec:intro}

Dobrinen and Simpson \cite{dob:04} asked how difficult it is to prove, in the context of reverse mathematics, the following three 
statements about the Lebesgue measure $\mu$ on $2^{\omega}$.  (The reader who is not familiar with the project of reverse mathematics 
is referred to Simpson \cite{si:09} for an introduction to the subject.)
\begin{enumerate}
\item[(1)] $\Gdelta$:  For every $G_{\delta}$ set $P \subseteq 2^{\omega}$, there is an $F_{\sigma}$ set $Q \subseteq P$ such that 
$\mu(Q) = \mu(P)$.
\item[(2)] $\Gepsilon$:  For every $G_{\delta}$ set $P \subseteq 2^{\omega}$ and every $\epsilon > 0$, there is a closed set $F \subseteq P$ 
such that $\mu(F) \geq \mu(P) - \epsilon$.
\item[(3)] $\Positive$:  For every $G_{\delta}$ set $P \subseteq 2^{\omega}$ such that $\mu(P) > 0$, there is a closed set $F \subseteq P$ 
such that $\mu(F) > 0$. 
\end{enumerate}

It is straightforward to show that $\ACA$ proves all three statements, $\RCA \vdash \Gdelta \rightarrow \Gepsilon$ and 
$\RCA \vdash \Gepsilon \rightarrow \Positive$.  Dobrinen and Simpson introduced the notions of \emph{uniformly almost everywhere} (\emph{u.a.e.}) \emph{domination} and \emph{almost everywhere} (\emph{a.e.}) \emph{domination} and showed that these are the recursion theoretic counterparts of 
$\Gdelta$ and $\Gepsilon$.    

\begin{defn}[Dobrinen and Simpson \cite{dob:04}] 
A set $A\in 2^\omega$ is \emph{a.e.~dominating} 
if for almost all $X \in 2^{\omega}$ (with respect to the Lebesgue measure) and all functions $g \leq_T X$, there is a function 
$f \leq_T A$ such that $f$ dominates $g$ (that is, $\exists m \forall n > m \, (f(n) \geq g(n))$).  
$A \in 2^{\omega}$ is \emph{u.a.e.~dominating} if there is a single function $f \leq_T A$ 
such that for almost all $X \in 2^{\omega}$ and all functions $g \leq_T X$, $f$ dominates $g$.
\end{defn}

\begin{thm}[Dobrinen and Simpson \cite{dob:04}] 
\label{thm:gdeltaequiv}
The following are equivalent.
\begin{enumerate}
\item $A$ is u.a.e.~dominating.
\item For all $\Pi^0_2$ sets $P \subseteq 2^{\omega}$, there is a $\Sigma^A_2$ set $Q \subseteq P$ such that 
$\mu(Q) = \mu(P)$.
\end{enumerate}
\end{thm}

\begin{thm}[Dobrinen and Simpson \cite{dob:04}] 
\label{thm:aeequiv}
The following are equivalent.
\begin{enumerate}
\item $A$ is a.e.~dominating.
\item For all $\Pi^0_2$ sets $P \subseteq 2^{\omega}$ and all $\epsilon > 0$, there is a $\Pi^A_1$ set $F \subseteq P$ 
such that $\mu(F) \geq \mu(P) - \epsilon$.
\end{enumerate}
\end{thm}

Dobrinen and Simpson observed that $\WKL \nvdash \Gdelta$ and asked whether any (or all) of $\Gdelta$, $\Gepsilon$ 
or $\Positive$ implied $\ACA$.  They suggested finding simpler recursion theoretic equivalences of a.e.~domination and 
u.a.e.~domination to help answer this question.  At that time, it was known that 
\[
A \text{ is complete } (A \geq_T \emptyset')  \Rightarrow A \text{ is u.a.e.~dominating } \Rightarrow A \text{ is high } (A' \geq_T \emptyset'').
\]
The first implication is a result of Kurtz \cite{kur:81} while the second implication follows from Martin's Theorem \cite{mar:66}.  
Dobrinen and Simpson asked 
whether either of these implications reverses.  Cholak, Greenberg and Miller \cite{cho:06} proved that the first arrow does not 
reverse and that even $\Gdelta$, the strongest of the measure theoretic statements, does not imply $\ACA$.  

\begin{thm}[Cholak, Greenberg and Miller \cite{cho:06}]
There is a (c.e.) set $A <_T \emptyset'$ such that $A$ is u.a.e.~dominating (and hence u.a.e.~domination does not imply completeness).  
Furthermore, $\WKL + \Gdelta$ does not imply $\ACA$, and $\RCA + \Gdelta$ does not imply the much weaker 
principle $\DNR$.  
\end{thm}

Binns, Kjos-Hanssen, Lerman and Solomon \cite{bin:06} proved that the second arrow does not reverse by constructing 
a high c.e.~set $A$ which is not a.e.~dominating. In addition, they found a connection between a.e.~domination 
and randomness, specifically the reducibility $\leq_\LR$ developed by Nies \cite{nie:05}.  

There are several ways to formalize algorithmic randomness and we start with a measure theoretic approach due to 
Martin-L\"{o}f.  A \emph{Martin-L\"{o}f test relative to an oracle} $A$ is an $A$-computable sequence of nested $\Sigma^A_1$ 
classes $U^A_0 \supseteq U^A_1 \supseteq \cdots$ such that $\mu(U^A_n) \leq 2^{-n}$.  
A set $R$ is $A$-\emph{random} if for every Martin-L\"{o}f test relative to $A$, 
$R \notin \bigcap_{n \in \omega} U^A_n$.  This notion of randomness is often called Martin-L\"{o}f randomness (relative to $A$) or $1$-randomness (relative to $A$).

\begin{defn}[Nies \cite{nie:05}]
	$A\leq_\LR B$ if every $B$-random real is $A$-random.
\end{defn}

The idea of $A \leq_\LR B$ is that $A$ is no more useful than $B$ in the sense that $A$ does not ``derandomize" any 
$B$-random sets.  

\begin{thm}[Binns, Kjos-Hanssen, Lerman and Solomon \cite{bin:06}]
\label{thm:aerandom}
If $A$ is a.e.~dominating, then $\emptyset' \leq_\LR A$. 
\end{thm}

Applying work of Nies \cite{nie:05}, it follows from Theorem \ref{thm:aerandom} that if $A \leq_T \emptyset'$ is a.e.~dominating, 
then $A$ is high, in fact superhigh (namely, $\emptyset'' \leq_{tt} A'$). Using the methods introduced in the present paper, Simpson \cite{sim:07} has generalized this corollary by removing the restriction that $A \leq_T \emptyset'$.  

The proof of Theorem \ref{thm:aerandom} actually shows that 
$\emptyset' \leq_{LR} A$ follows from the assumption that for every $\Pi^0_2$ class $P \subseteq 2^{\omega}$ such that 
$\mu(P) > 0$, there is a $\Pi^A_1$ class $Q \subseteq P$ such that $\mu(Q) > 0$.  (This property is the recursion 
theoretic analogue of $\Positive$.)  Kjos-Hanssen proved that this property is equivalent to what he called \emph{positive measure} (\emph{p.m.}) \emph{domination} and 
proved the following general theorem connecting $\leq_\LR$ with the ability to find closed subclasses of positive measure.

\begin{thm}[Kjos-Hanssen \cite{kjo:07}]
\label{thm:bjorn}
	$A\leq_\LR B$ if and only if every $\Pi^A_1$ class of positive measure has a $\Pi^B_1$ subclass of positive measure.
\end{thm}

Combining Theorem \ref{thm:bjorn} with the well-known result of Kurtz \cite{kur:81} that every $\Pi^0_2$ class has a 
$\Sigma^{\emptyset'}_2$ subclass of the same measure, it follows that $\emptyset' \leq_\LR A$ exactly characterizes the p.m.~dominating sets.

\begin{cor}[Kjos-Hanssen \cite{kjo:07}]
\label{cor:bjorn}
$A$ is p.m.~dominating if and only if $\emptyset' \leq_\LR A$.
\end{cor}

As this point, we have the following picture.
\begin{gather*}
A \text{ is u.a.e.~dominating} \Rightarrow A \text{ is a.e.~dominating}  \\
\Rightarrow A \text{ is p.m.~dominating} \Leftrightarrow \emptyset' \leq_\LR A
\end{gather*}
In Section \ref{sec:preserving}, we close this circle by showing that if $A$ is p.m.~dominating, then $A$ is u.a.e.~dominating. This result is an application of a more general theorem along the lines of Theorem \ref{thm:bjorn}: every $\Sigma^A_2$ class has a $\Sigma^B_2$ subclass of the same measure if and only if $A \leq_\LR B$ and $A \leq_T B'$. As another application, we prove that if $A$ is low for $1$-randomness then it is low for weak $2$-randomness (see also Nies~\cite{ni:09}). The main technique used in Section~\ref{sec:preserving} gives us a new way to leverage the assumption that $A \leq_\LR B$. It is first introduced in Section~\ref{sec:randomred}, where we show that $\leq_\LR$ implies $\leq_\LK$, a reducibility that compares the strength of oracles in terms of their effect on prefix-free Kolmogorov complexity.

In the remaining sections, we examine the implication of the equivalence of u.a.e.~domination and p.m.~domination for the
reverse mathematics question of how difficult it is to prove that $\Positive \rightarrow \Gdelta$. In Section
\ref{sec:rec}, we show that $\RCA$ is not strong enough to prove this implication, or even that $\Gepsilon \rightarrow
\Gdelta$. In Section \ref{sec:wwkl}, we show that $\WWKL\vdash \Positive \rightarrow \Gdelta$. Notice that since $\WKL$
does not prove $\Gdelta$, the fact that $\WWKL$---which is weaker than $\WKL$---proves this implication is not trivial.
Moreover, since measure theory is very limited without $\WWKL$ \cite{yu:90}, it is reasonable to work over this system to
prove the equivalence.

Our notation is standard throughout. We use $\subseteq$ to denote the subset relation between sets (or classes),
$\sqsubseteq$ to denote the initial segment relation between (finite or infinite) strings, and $|\sigma|$ to denote the
length of a finite string $\sigma$. We identify a set $X$ with the infinite string given by its characteristic function.
For $X \subseteq \omega$ and $s \in \omega$, $X[s]$ denotes the string $\langle X(0), X(1), \ldots, X(s-1) \rangle$. For
$Y \subseteq 2^{< \omega}$, $[Y]$ denotes the open class in $2^{\omega}$ of all $X$ such that $\exists \sigma \in Y
(\sigma \sqsubseteq X)$. If $Z \subseteq 2^{\omega}$, then $Z^c = 2^{\omega} \setminus Z$. Finally, if $M$ is any machine
(viewed as defining a partial function from $2^{< \omega}$ to $2^{< \omega}$), then $\text{dom}(M)$ denotes the set of
strings on which $M$ converges (that is, the domain of the defined function).

\section{$\leq_\LR$ implies $\leq_\LK$}
\label{sec:randomred}

In this section, we examine the relationship between $\leq_\LR$ and $\leq_\LK$, a reducibility based on an information 
theoretic definition of randomness.  The reader who is not familiar with Kolmogorov complexity is referred to Li and 
Vit\'anyi \cite{livi:08} for an introduction.  If $U$ is a universal prefix-free (Turing) machine and $\tau$ is a finite binary string, then the \emph{prefix-free} (\emph{Kolmogorov}) \emph{complexity} of $\tau$ is defined (up to an additive constant depending on the choice of $U$) by 
\[
K(\tau) = \min \{ |\sigma| \mid U(\sigma) = \tau \}.
\]
We will use two basic facts from the theory of Kolmogorov complexity.  

\begin{lem}[Kraft inequality] 
If $A \subseteq 2^{< \omega}$ is prefix-free, then $\sum_{\sigma \in A} 2^{-|\sigma|} \leq 1$.  In particular, if $M$ is a 
prefix-free Turing machine, then $\sum_{\sigma \in \text{dom}(M)} 2^{-|\sigma|} \leq 1$.
\end{lem}

\begin{thm}[Kraft--Chaitin Theorem] 
Let $\langle d_i, \tau_i \rangle_{i \in \omega}$ be a computable sequence of pairs such that $d_i \in \omega$, $\tau_i
\in 2^{< \omega}$ and $\sum_{i \in \omega} 2^{-d_i} \leq 1$. (The range $\{\langle d_i,\tau_i\rangle : i\in\omega\}$ of such a sequence is called a Kraft--Chaitin set.) There is a
prefix-free machine $M$ and strings $\sigma_i$ of length $d_i$ such that $M(\sigma_i) = \tau_i$ for all $i \in \omega$.
In particular, the universality of $U$ implies that $K(\tau_i) \leq d_i + O(1)$.
\end{thm}

$A$ is called \emph{Levin-Chaitin random} if for all $n$, $K(A[n]) \geq n - O(1)$.  Despite the difference in context, this notion of 
randomness coincides with Martin-L\"{o}f randomness defined above.  Nies \cite{nie:05} defined a reducibility $\leq_\LK$ 
similar to $\leq_\LR$, but based on Kolmogorov complexity.  The idea of this reducibility is that $A \leq_\LK B$ if 
$A$ is no more useful than $B$ in the sense that $A$ cannot compress information any more than $B$ can.  

\begin{defn}[Nies \cite{nie:05}]
	$A\leq_\LK B$ if $(\forall \tau)\; K^B(\tau)\leq K^A(\tau)+O(1)$.
\end{defn}

It is straightforward to show that $A\leq_\LK B$ implies $A\leq_\LR B$; our goal for this section is to show that 
they are equivalent.  Our proof will require one basic fact from real analysis. 

\begin{lem}
\label{lem:analysis}
	Let $\seq{a_i}_{i\in\omega}$ be a sequence of real numbers with $0\leq a_i< 1$, for all
$i$. Then $\prod_{i\in\omega} (1-a_i)>0$ iff\/ $\sum_{i\in\omega} a_i$ converges.
\end{lem}

\begin{lem}\label{2009}
For any computable function $f:\omega\rightarrow\omega$ there is a uniformly computable collection of finite sets of binary strings $V_n$, $n\in\omega$, such that $\mu [V_n]=2^{-f(n)}$ and the sets $[V_n]$, $n\in\omega$, form a mutually independent family of events under $\mu$.
\end{lem}
\begin{proof}
Assume that $V_t$ has been defined for all $t<s$. Let $k$ be the length of the longest string in $\bigcup_{t<s}V_t$ and let $V_s=\{\sigma\cat 0^{f(s)}: \sigma\in 2^k\}$. It is clear that $V_s$, $s\in\omega$, has the required properties.
\end{proof}

\begin{thm}\label{thm:LR->LK}
	If $A\leq_\LR B$, then $A\leq_\LK B$.
\end{thm}
\begin{proof}
Identifying the elements of $\omega\times 2^{<\omega}$ with natural numbers via an effective bijection, we let $V_s$, $s\in\omega$ be as guaranteed by Lemma \ref{2009} for the function $f(\seq{n,\tau})=n$. This ensures that if
$I\subseteq\omega\times 2^{<\omega}$, then $\mu\left(\bigcap_{s\in I}[V_s] ^c\right) = \prod_{\seq{n,\tau}\in I}
(1-2^{-n})$, since each $V_s$ is independent from all of the others.

Let $U^A$ be a universal prefix-free machine relative to $A$ and define 
\[
I = \{\seq{|\sigma|,\tau}\colon
U^A(\sigma)=\tau\}.
\]
Then $I$ is $A$-c.e., so $P = \bigcap_{s\in I}[V_s]^c$ is a $\Pi^A_1$ class. Note that
$\sum_{\seq{n,\tau}\in I} 2^{-n} \leq \sum_{\sigma\in\dom(U)} 2^{-|\sigma|}\leq 1$ by the Kraft inequality. Also,
$\seq{0,\tau}$ is not in $I$ for any $\tau$. So by Lemma~\ref{lem:analysis}, $\mu(P) = \prod_{\seq{n,\tau}\in I}
(1-2^{-n}) > 0$. Therefore by Theorem~\ref{thm:bjorn}, there is a $\Pi^B_1$ class $Q\subseteq P$ such that $\mu(Q)>0$.

Define $J = \{\seq{n,\tau}\colon [V_\seq{n,\tau}]\cap Q = \emptyset\}$. Note that $J$ is a $B$-c.e.~set since $Q^c$ is generated by a
$B$-c.e.~set of strings, $V_\seq{n,\tau}$ is a finite set of strings, and $[V_\seq{n,\tau}]\cap Q = \emptyset$ if and only
if $[V_\seq{n,\tau}]$ is covered by a finite set of basic intervals from $Q^c$. Also, by the comments in the first paragraph of
this proof, $\prod_{\seq{n,\tau}\in J} (1-2^{-n}) = \mu\left(\bigcap_{s\in J}[V_s]^c\right)\geq \mu(Q)>0$. Therefore by
Lemma~\ref{lem:analysis}, $\sum_{\seq{n,\tau}\in J} 2^{-n}$ converges. Furthermore, we claim that $I \subseteq J$. If
$\seq{n,\tau} \in I$, then $[V_{\seq{n,\tau}}] \cap P = \emptyset$. Since $Q \subseteq P$, $[V_{\seq{n,\tau}}] \cap Q =
\emptyset$ and hence $\seq{n,\tau} \in J$.

Since $\sum_{\seq{n,\tau}\in J} 2^{-n}$ converges, fix $c \in \omega$ such that this sum is bounded by $2^c$. Then
$\widehat{J} = \{\seq{n+c,\tau}\colon \seq{n,\tau}\in J\}$ is a Kraft--Chaitin set relative to $B$. Therefore by the
Kraft--Chaitin Theorem, 
\[
	\seq{n,\tau}\in J \implies \seq{n+c,\tau}\in \widehat{J} \implies K^B(\tau)\leq n+c+O(1)\leq
n+O(1).
\]
Since $I \subseteq J$, we have $\seq{K^A(\tau),\tau}\in J$ for each $\tau\in 2^\omega$. Thus $K^B(\tau)\leq
K^A(\tau)+O(1)$. In other words, $A\leq_\LK B$.
\end{proof}

\begin{cor}
\label{cor:equiv}
$A \leq_\LR B$ if and only if $A \leq_\LK B$.
\end{cor}

\begin{proof}
As noted previously, $A \leq_\LK B$ implies $A \leq_\LR B$.  Theorem \ref{thm:LR->LK} supplies the other implication.  
\end{proof}

We offer one application of Theorem \ref{thm:LR->LK} based on a special case of $\leq_\LR$ and $\leq_\LK$.  $A$ is 
\emph{low for $1$-randomness} if $A \leq_\LR \emptyset$, that is, if every random (in the measure theoretic sense) remains 
random relative to $A$.  Similarly, $A$ is called \emph{low for K} if $A \leq_\LK \emptyset$, that is, every string contains
as much information relative to $A$ as it does with no oracle.

\begin{cor}[Nies \cite{nie:05}\footnote{Yet another proof---one based on work of Hirschfeldt, Nies and Stephan \cite{hir:07}---can be found in Nies \cite{ni:09}.}]
	$A$ is low for $1$-randomness if and only if $A$ is low for $K$.
\end{cor}

\begin{proof}
This corollary follows from Corollary \ref{cor:equiv} by setting $B = \emptyset$.
\end{proof}

\section{Preserving Measure}
\label{sec:preserving} 

In this section, we show that p.m.~domination implies u.a.e.~domination, thereby showing the equivalence of the three domination 
notions introduced in Section \ref{sec:intro}.  

\begin{lem}\label{lem:measure}
	If $A\leq_T B'$ and $A\leq_\LR B$, then every $\Pi^A_1$ class has a $\Sigma^B_2$ subclass of the same measure. 
\end{lem}
\begin{proof}
The proof will be similar to that of Theorem~\ref{thm:LR->LK}. Identifying now the elements of $2^{<\omega}\times 2^{<\omega}$ with natural numbers via an effective bijection, we let $\{V_s\}_{s\in\omega}$ be as guaranteed by Lemma \ref{2009} for the function $f(\seq{\sigma,\tau})=|\tau|$. As before, if 
$I\subseteq2^{<\omega}\times 2^{<\omega}$, then $\mu\left(\bigcap_{s\in I}[V_s]^c\right) = \prod_{\seq{\sigma,\tau}\in I} (1-2^{-|\tau|})$.

Let $X$ be a $\Pi^A_1$ class. Assume, without loss of generality, that $X\neq\emptyset$. Let $S^A\subseteq 2^{<\omega}$ be a prefix-free 
$A$-c.e.\ set of strings such that $X=2^\omega\smallsetminus[S^A]$; note that $S^A$ does not contain the empty string. 
Let $I = \{\seq{\sigma,\tau}\colon \tau\in S^A\text{ with use }\sigma\}$. Consider the $\Pi^A_1$ class $P = \bigcap_{s\in I}[V_s]^c$. 
Note that $\sum_{\seq{\sigma,\tau}\in I} 2^{-|\tau|} = \sum_{\tau\in S^A} 2^{-|\tau|}\leq 1$ by the Kraft inequality. So by Lemma~\ref{lem:analysis}, 
$\mu(P) = \prod_{\seq{\sigma,\tau}\in I} (1-2^{-|\tau|}) > 0$. Therefore by Theorem~\ref{thm:bjorn}, there is a $\Pi^B_1$ class $Q\subseteq P$ 
such that $\mu(Q)>0$.

Define $J = \{\seq{\sigma,\tau}\colon [V_\seq{\sigma,\tau}]\cap Q = \emptyset\}$.  As in the proof of Theorem \ref{thm:LR->LK}, 
$J$ is a $B$-c.e.~set, $I \subseteq J$, and 
$\prod_{\seq{\sigma,\tau}\in J} (1-2^{-|\tau|}) = \mu\left(\bigcap_{s\in J}[V_s]^c\right)\geq \mu(Q)>0$. Therefore by 
Lemma~\ref{lem:analysis}, $\sum_{\seq{\sigma,\tau}\in J} 2^{-|\tau|}$ converges.    

By assumption $A \leq_T B'$, so let 
$\{A_s\}_{s\in\omega}$ be a $B$-computable sequence approximating $A$. Define 
\[ 
T_s = \{\seq{\sigma,\tau}\in J\colon (\exists t\geq s)\; \tau\in S^{A_t}_t\text{ with use }\sigma\}
\] 
and let $U_s = \{\tau\colon (\exists \sigma)\; \seq{\sigma,\tau}\in T_s\}$ be the projection of $T_s$ onto the second co\,ordinate.  
$\{T_s\}_{s\in\omega}$ and $\{U_s\}_{s\in\omega}$ are $B$-computable (nested) sequences of $B$-c.e.\ sets. We claim that 
$Y = \bigcup_{s\in\omega} [U_s]^c$ is the desired $\Sigma^B_2$ class.

We claim that $S^A\subseteq U_s$ for all $s$, so $Y\subseteq X$.  Suppose $\tau \in S^A$ and fix the use $\sigma$ of this 
computation.  Then $\langle \sigma, \tau \rangle \in I$ and hence $\langle \sigma, \tau \rangle \in J$.  Because $A_s$ is a $B$-computable 
approximation to $A$, it follows that $\forall s \exists t \geq s (\tau \in S^{A_t}_t \text{ with use } \sigma)$.  
In other words, $\langle \sigma, \tau \rangle \in T_s$ for all $s$, and hence $\tau \in U_s$ for all $s$ as required.  

For each $\seq{\sigma,\tau}\in T_0\smallsetminus I$, there is a last stage $t$ such that $\sigma$ is a prefix of $A_t$, otherwise 
$\seq{\sigma,\tau}$ would be in $I$. Then $\seq{\sigma,\tau}\notin T_s$ for any $s>t$. Fix $\epsilon>0$. Take $n$ large enough that 
$\sum_{\seq{\sigma,\tau}\in J,\, \seq{\sigma,\tau}\geq n} 2^{-|\tau|} < \epsilon$ and take $s$ large enough that 
$\seq{\sigma,\tau}\in T_0\smallsetminus I$ and $\seq{\sigma,\tau}< n$ implies $\seq{\sigma,\tau}\notin T_s$. Then, 
\[ 
\mu(X\smallsetminus[U_s]^c)\leq \sum_{\tau\in U_s\smallsetminus S^A} 2^{-|\tau|}\leq \sum_{\seq{\sigma,\tau}\in 
T_s\smallsetminus I} 2^{-|\tau|}\leq \sum_{\seq{\sigma,\tau}\in J,\, \seq{\sigma,\tau}\geq n} 2^{-|\tau|} < \epsilon. 
\] 
But $\epsilon>0$ was arbitrary, so $\mu(X) = \mu(Y)$.
\end{proof}

\begin{thm}
\label{thm:equal}
	The following are equivalent:
	\begin{enumerate}
		\item $A\leq_T B'$ and $A\leq_\LR B$,
		\item Every $\Pi^A_1$ class has a $\Sigma^B_2$ subclass of the same measure,
		\item Every $\Sigma^A_2$ class has a $\Sigma^B_2$ subclass of the same measure.
	\end{enumerate}
\end{thm}
\begin{proof}
	(i)$\implies$(ii) is Lemma~\ref{lem:measure}.

(ii)$\implies$(iii): Let $W$ be a $\Sigma^A_2$ class. So $W = \bigcup_{i\in\omega} X_i$ for $\Pi^A_1$ classes $\{X_i\}_{i\in\omega}$. 
Consider the $\Pi^A_1$ class $X = \{0^i1\cat\alpha\colon i\in\omega\text{ and }\alpha\in X_i\}$. By (ii), there is a $\Sigma^B_2$ class 
$Y\subseteq X$ such that $\mu(Y)=\mu(X)$. For each $i$, let $Y_i = \{\alpha\colon 0^i1\cat\alpha\in Y\}$. So, $Y_i$ is a $\Sigma^B_2$ 
class and $Y_i\subseteq X_i$ for all $i$. Clearly $\mu(Y_i)\leq\mu(X_i)$. If $\mu(Y_i)<\mu(X_i)$ for some $i$, then 
$\mu(Y)=\sum_{i\in\omega}2^{i+1}\mu(Y_i)< \sum_{i\in\omega}2^{i+1}\mu(X_i)=\mu(X)$, which is a contradiction. Therefore, 
$\mu(Y_i) = \mu(X_i)$ for all $i$. Let $Z = \bigcup_{i\in\omega} Y_i$. So $Z$ is a $\Sigma^B_2$ class and $Z\subseteq W$. Furthermore, 
$\mu(W\smallsetminus Z)\leq \sum_{i\in\omega}\mu(X_i\smallsetminus Y_i) = 0$, so $\mu(Z)=\mu(W)$.

(iii)$\implies$(i): Suppose that every $\Sigma^A_2$ class has a $\Sigma^B_2$ subclass of the same measure.  First, we show that 
$A \leq_\LR B$.  By Theorem \ref{thm:bjorn}, it suffices to show that if $P$ is a $\Pi^A_1$ class of positive measure, then $P$ has a 
$\Pi^B_1$ subclass of positive measure.  By assumption, $P$ has a $\Sigma^B_2$ subclass $Q = \bigcup_{i \in \omega} Q_i$ of 
positive (in fact the same) measure.  At least one of the $\Pi^B_1$ classes $Q_i \subseteq Q \subseteq P$ must have positive 
measure.  

Next, we show that $A \leq_T B'$.  Let $\sigma_n = 0^n1$ and consider the $\Sigma^A_1$ class $U = \bigcup_{n \in A} [\sigma_n]$.  
Since $U$ is a $\Sigma^A_1$ (and hence a $\Pi^A_2$) class, by (iii) 
there is a $\Pi^B_2$ class $Q$ such that $U \subseteq Q$ and $\mu(Q) = \mu(U) = \sum_{n \in A} 
2^{-(n+1)}$.  We claim that $n \in A$ if and only if $[\sigma_n] \subseteq Q$.  If $n \in A$, then $[\sigma_n] \subseteq U \subseteq Q$.  
On the other hand, if $n \notin A$ and $[\sigma_n] \subseteq Q$, then $\mu(Q) \geq \sum_{i \in A} 2^{-(i+1)} + 2^n > \mu(U)$ which 
is a contradiction.  Writing $Q = \bigcap_{k \in \omega} Q_k$ where each $Q_k$ is $\Sigma^B_1$, we have 
\[
n \in A \, \Leftrightarrow \, [\sigma_n] \subseteq Q \,\Leftrightarrow \forall k ([\sigma_n] \subseteq Q_k).
\]
Since $[\sigma_n] \subseteq Q_k$ is a $\Sigma^B_1$ relation, these equivalences show that $A$ is $\Pi^B_2$.  However, the same 
argument with the $\Sigma^A_1$ class $\bigcup_{n \notin A} [\sigma_n]$ shows that $\overline{A}$ is $\Pi^B_2$ as well, 
and hence $A \leq_T B'$.   
\end{proof}

We cannot remove the condition that $A \leq_T B'$ from Theorem \ref{thm:equal}. Indeed, there is a $B$ for which uncountably many $A$ satisfy $A\le_{LR}B$ (see Barmpalias, Lewis, and Soskova \cite{bar:2008}), whereas for each $B$ there are only countably many $A$ with $A \leq_T B'$.  

\begin{cor}
\label{cor:pmd}
For all $B$, the following are equivalent:
\begin{enumerate}
\item[(1)] $B$ is uniformly almost everywhere dominating,
\item[(2)] $B$ is almost everywhere dominating, 
\item[(3)] $B$ is positive measure dominating, and 
\item[(4)] $\emptyset' \leq_\LR B$.
\end{enumerate}
\end{cor}

\begin{proof}
As noted in Section \ref{sec:intro}, 
we have (1) implies (2), (2) implies (3), and (3) if and only if (4).  It remains to show that (4) implies (1).  
Suppose $\emptyset' \leq_\LR B$. Since $\emptyset' \leq_T B'$,
Theorem \ref{thm:equal} tells us that every $\Sigma^{\emptyset'}_2$ class has a $\Sigma^B_2$ subclass of the same measure.  
By Theorem \ref{thm:gdeltaequiv}, to show that $B$ is uniformly almost everywhere dominating, it suffices to show that 
every $\Pi^0_2$ class has a $\Sigma^B_2$ subclass of the same measure.  Fix a $\Pi^0_2$ class $P$.  By 
Kurtz \cite{kur:81}, $P$ contains a $\Sigma^{\emptyset'}_2$ 
subclass $\hat{P}$ such that $\mu(\hat{P}) = \mu(P)$.  But, $\hat{P}$ contains a 
$\Sigma^B_2$ subclass $Q$ of the same measure and hence $Q \subseteq \hat{P} \subseteq P$ and $\mu(Q) = \mu(\hat{P}) 
= \mu(P)$ as required.
\end{proof}

Our second corollary of Theorem \ref{thm:equal} involves the notions of low for weak $2$-randomness and low for weak 
$2$-random tests.  A \emph{generalized Martin-L\"{o}f test} is a computable nested 
sequence of $\Sigma^0_1$ classes $U_0 \supseteq U_1 \supseteq \cdots$ such that $\mu(\bigcap_{i \in \omega} U_i) = 0$.  That is, 
a generalized Martin-L\"{o}f test is a Martin-L\"{o}f test with the restriction that $\mu(U_i) \leq 2^{-i}$ loosened.
Note that if $\{U_i\}_{i\in\omega}$ is a generalized Martin-L\"{o}f test, then $\bigcap_{i \in \omega} U_i$ is a
$\Pi^0_2$ class of measure $0$, and conversely, that any $\Pi^0_2$ class of measure $0$ can be viewed as a
generalized Martin-L\"{o}f test.  A set $X$ is 
\emph{weakly $2$-random} if $X \notin \bigcap_{i \in \omega} U_i$ for all generalized Martin-L\"{o}f tests.  
Notice that all weakly $2$-random sets are $1$-random.

We say that $A$ is \emph{low for weak $2$-randomness} if every set $X$ that is weakly $2$-random is also weakly $2$-random relative to $A$.  In other words, if $X \notin \bigcap_{i \in \omega} U_i$ for all generalized Martin-L\"{o}f tests, then $X \notin \bigcap_{i \in \omega} V_i^A$ for all generalized Martin-L\"{o}f tests relative to $A$.  Because weak $2$-randomness has been defined in 
terms of tests, it is possible to give a more uniform version of this condition.  $A$ is \emph{low for weak $2$-random tests} if for 
every generalized Martin-L\"{o}f test $\bigcap_{i \in \omega} V_i^A$ relative to $A$, there is a generalized Martin-L\"{o}f test 
$\bigcap_{i \in \omega} U_i$ such that $\bigcap_{i \in \omega} V_i^A \subseteq \bigcap_{i \in \omega} U_i$.  It follows immediately 
that if $A$ is low for weak $2$-random tests, then $A$ is low for weak $2$-randomness.  

\begin{cor}
\label{cor:weak}
If $A$ is low for $1$-randomness, then $A$ is low for weak $2$-random tests.
\end{cor}

\begin{proof}
Suppose that $A$ is low for $1$-randomness, that is, $A \leq_\LR \emptyset$.  Since every low for $1$-random set is low (that is, 
$A' \leq_T \emptyset'$, in fact, even $A' \leq_{tt} \emptyset'$), $A$ satisfies the conditions in Theorem \ref{thm:equal}(i) with $B = \emptyset$.  
Therefore, every $\Sigma^A_2$ class has a $\Sigma^0_2$ subclass of the same measure.  In particular, every 
$\Pi^A_2$ class of measure $0$ is contained in a $\Pi^0_2$ class of measure $0$.  In other words, every generalized 
Martin-L\"{o}f test relative to $A$ is contained in a generalized Martin-L\"{o}f test as required.
\end{proof}

Downey, Nies, Weber and Yu \cite{dow:06} proved one implication between low for $1$-randomness and low for 
weak $2$-randomness.

\begin{thm}[Downey, Nies, Weber and Yu \cite{dow:06}]
\label{thm:weak}
	If $A$ is low for weak $2$-randomness, then $A$ is low for $1$-randomness.
\end{thm}

Combining Corollary \ref{cor:weak} and Theorem \ref{thm:weak} together with the fact that low for weak $2$-random tests 
implies low for weak for $2$-randomness yields the following corollary.  

\begin{cor}
\label{cor:combine}
For any set $A$, the following conditions are equivalent:
\begin{enumerate}
\item[(1)] $A$ is low for $1$-randomness, 
\item[(2)] $A$ is low for weak $2$-random tests, and
\item[(3)] $A$ is low for weak $2$-randomness.
\end{enumerate}
\end{cor}

Corollary \ref{cor:combine} can also be proved using the \emph{golden run} machinery of Nies \cite{nie:05}. This was discovered 
independently, and earlier, by Nies and a proof along these lines is given in Nies \cite{ni:09}.

\section{Measure Definitions in Reverse Mathematics}
\label{sec:measuredef}

In the remainder of this paper, we consider the reverse mathematics question of how difficult it is to prove $\Positive \rightarrow \Gdelta$. We begin with definitions of codes for open, closed, $G_{\delta}$ and $F_{\sigma}$ subsets of $2^{\mathbb{N}}$ in \RCA. (We switch from $\omega$ to $\mathbb{N}$ as it is standard to use $\mathbb{N}$ to denote the first order part of any given model of second order arithmetic.)

A \textit{code for an open set in} $2^{\mathbb{N}}$ is a set $O \subseteq 2^{< \mathbb{N}}$.  We can 
assume without loss of generality that $O$ is prefix free.  
We write $X \in [O]$ (and say that $X$ is in the set coded by $O$) if 
there is a string $\tau \in O$ such that $t \sqsubseteq X$.  It is often useful to think of an open set 
as the union of a sequence of clopen sets.  For $t \in \mathbb{N}$, we let 
$O_t = \{ \tau \in O \mid |\tau| < t \}$ and note that $[O] = \bigcup_t [O_t]$.  

Equivalently, we can specify an open set by a $\Sigma^0_1$ formula (allowing parameters) $\exists s \varphi(x)$, where $\varphi(x)$ contains only bounded 
quantifiers.  In this context, we say that $X$ is in the coded open set if $\exists s \varphi(X[s])$.  Later it 
will be convenient to think of the collection of strings 
satisfying (or enumerated by) such a formula even though this collection need not be a set in $\RCA$.  We use the term $\Sigma^0_1$ class of strings (or simply 
$\Sigma^0_1$ class, relying on context to differentiate between this notion of class and the one used in the context of sets of reals) to denote 
the collection of strings satisfying a particular $\Sigma^0_1$ formula.  This terminology allows us to use set notation for such collections, although any such 
statement is understood as standing for the appropriate translation of the defining formulas.  If $O$ is 
the $\Sigma^0_1$ class of strings corresponding to the formula $\exists s \varphi(x)$, then 
$O_t = \{ \tau \mid |\tau| < t \wedge \exists s < t \, \varphi(\tau) \}$.  As above, each $O_t$ is clopen and 
$[O] = \bigcup_t [O_t]$.  In this context, we cannot assume that the $\Sigma^0_1$ class of strings $O$ is 
prefix free.  However, abusing notation, we can assume (by removing strings from $O_t$ in a 
uniform manner) that the finite sets $O_t$ are prefix free.

In systems weaker than $\ACA$, we cannot assume that bounded increasing 
sequences of rationals converge.  Therefore, rather than assuming that open sets 
have definite measures, we work with comparative statements such as $\mu(O) \geq q$ for 
$q \in \mathbb{Q}$.  To define these notions in \RCA, 
let $O$ be a (prefix-free) code for an open set.  For $t \in \mathbb{N}$, define 
$\mu(O_t) = \sum_{\tau \in O_t} 2^{-|\tau|}$, and for $q \in \mathbb{Q}$, define 
\begin{gather*}
\mu(O) \leq q \, \Leftrightarrow \, \forall t \, (\mu(O_t) \leq q) \\
\mu(O) > q \, \Leftrightarrow \, \exists t \, (\mu(O_t) > q) \\
\mu(O) \geq q \, \Leftrightarrow \, \forall r \in \mathbb{Q} (r < q \rightarrow \mu(O) > r)
\end{gather*}
Thus, $\mu(O) \leq q$ is a $\Pi^0_1$ statement (with parameter $O$), $\mu(O) > q$ is a $\Sigma^0_1$ statement, and $\mu(O) \geq q$ is a 
$\Pi^0_2$ statement.   However, if $\lim_{t \rightarrow \infty} \mu(O_t)$ is irrational, 
then $\mu(O) \geq q \Leftrightarrow \mu(O) > q$, and hence $\mu(O) \geq q$ is a 
$\Sigma^0_1$ expression.  

We specify a \emph{closed set} by giving a code $O$ for its complement as an open set and we write 
$X \in [O]^c$ if for all $\tau \in O$, $\tau \not \sqsubseteq X$.  (Equivalently, we can 
specify a closed set by a $\Pi^0_1$ formula $\forall s \varphi(x)$ and say that $X$ is in the closed set if 
$\forall s \varphi(X[s])$.)  
We say $\mu([O]^c) \geq q$ if $\mu([O]) \leq 1-q$, and similarly for the other inequalities. 

A \emph{code for a} $G_{\delta}$ set is a sequence $G = \langle G_k \mid k \in \mathbb{N} \rangle$ 
such that each $G_k$ is a code for an open set and we write $X \in [G]$ if  
for every $k$, there is a string $\tau_k \in G_k$ such that $\tau_k \sqsubseteq X$.  We 
frequently abuse notation and simply write $G = \bigcap_{k \in \mathbb{N}} G_k$.  (Equivalently, 
we can specify a $G_{\delta}$ set by a $\Pi^0_2$ formula $\forall n \exists s \varphi(x)$ and say that 
$X$ is in the coded set if $\forall n \exists s \varphi(X[s])$.)  

To define our measure 
inequalities for $G$, we form the sequence of open sets $\langle G^n \mid n \in 
\mathbb{N} \rangle$ where $G^n = \bigcap_{k = 0}^{n} G_k$.  Notice that $G^1 \supseteq 
G^2 \supseteq \cdots$ and that classically, $\mu(G) = \lim_n \mu(G^n)$.  For all $q \in 
\mathbb{Q}$, we define 
\begin{gather*}
\mu(G) \leq q \, \Leftrightarrow \, \forall r \in \mathbb{Q} ( r > q \rightarrow \exists n (\mu(G^n) \leq r) \\
\mu(G) \geq q \, \Leftrightarrow \, \forall n (\mu(G^n) \geq q)
\end{gather*}
Thus, $\mu(G) \leq q$ is a $\Pi^0_3$ statement and $\mu(G) \geq q$ is a $\Pi^0_2$ 
statement.    However, if $\lim_{n \rightarrow \infty} 
\mu(G^n)$ is irrational, then $\mu(G) \leq q \Leftrightarrow \exists n (\mu(G^n) \leq q)$ 
and hence $\mu(G) \leq q$ is a $\Sigma^0_2$ statement.  

A \emph{code for an} $F_{\sigma}$ set is also sequence $F = \langle F_n \mid n \in \mathbb{N} \rangle$ 
such that each $F_n$ is a code for an open set.  $F$ codes the union of the closed 
sets $[F_n]^c$: $X \in [F]$ if there is an $n$ such that $X \in [F_n]^c$.  (Equivalently, we can specify an 
$F_{\sigma}$ set by a $\Sigma^0_2$ formula $\exists n \forall s \varphi(x)$ and say that $X$ is in the coded set if 
$\exists n \forall s (\varphi(X[s]))$.)  We define the measure inequalities for 
an $F_{\sigma}$ set from the measure inequalities for its $G_{\delta}$ complement.  

When working in subsystems below \ACA, we regard a measure theoretic statement such as $\mu(G) = \mu(F)$ 
as an abbreviation for the sentence stating that for all $q \in \mathbb{Q}$, $\mu(G) \geq q$ if and only if 
$\mu(F) \geq q$.  That is, we do not assume that the measures converge to reals in the models for the weak subsystems.    

\section{Working in \REC}
\label{sec:rec}

In this section we work in $\REC$, the $\omega$-model consisting of the computable sets. A $G_{\delta}$ set in this model is called a \emph{computable} $G_{\delta}$ set. Our goal is to show that $\REC \notmodels \Gepsilon \rightarrow \Gdelta$ and hence that $\RCA\nvdash \Gepsilon \rightarrow \Gdelta$. Therefore, $\RCA\nvdash \Positive \rightarrow \Gdelta$.

First we show that $\REC \notmodels \Gdelta$. This follows from the existence of a computable $G_\delta$ with measure different from that of every computable $F_\sigma$ set, which in turn, follows easily from the existence of a set that is $\Pi^0_2$ but not $\Sigma^0_2$. Recall that if $G$ is a computable $G_{\delta}$ set and $q\in\Q$, then $\mu(G) \geq q$ is a $\Pi^0_2$ statement.

\begin{prop}
\label{prop:gdelta}
There is a computable $G_\delta$ set $G$ such that $\{q\in\Q\mid \mu(G)\geq q\}$ is not $\Sigma^0_2$.
\end{prop}
\begin{proof}
Let {\sc Tot} denote the $\Pi^0_2$ complete index set $\{e\in\omega\mid W_e=\omega\}$, where $\{W_e\}_{e\in\omega}$ is the standard enumeration of the c.e.\ sets. We identify {\sc Tot} with its characteristic function. Let $r = \sum_{i=0}^{\infty} \frac{\text{\sc Tot}(i)}{2^{i+1}}$, so the binary expansion of $r$ is {\sc Tot}.

Let $\le_L$ denote lexicographic order on $2^{\leq\omega}$. Define $G = \{X\in 2^\omega\mid X\leq_L\text{\sc Tot}\}$ and note that $r=\mu(G)$. To see that $G$ is a computable $G_\delta$ set, notice that
\[
X\in G \iff \forall n \exists s (X[n]\le_L \text{\sc Tot}_{n,s})
\]
where {\sc Tot}$_{n,s} = \{e<n\mid 0,\dots,n-1\in W_{e,s}\}$.

Now let $A = \{q\in\Q\mid \mu(G)\geq q\} = \{q\in\Q\mid r\geq q\}$. It is not hard to see that we can recover {\sc Tot} from $A$. First, note that $0\in\text{\sc Tot}$ if and only if $1/2\in A$ (using the fact that {\sc Tot} is coinfinite). Next, $1\in\text{\sc Tot}$ if and only if either $0\in\text{\sc Tot}$ and $3/4\in A$ or $0\notin\text{\sc Tot}$ and $1/4\in A$. The induction continues in the obvious way, showing that $\text{\sc Tot}\leq_T A$.

As noted above, $A$ is a $\Pi^0_2$ set. If $A$ were $\Sigma^0_2$, then $A$ would be computable from $\emptyset'$. But this would imply that $\emptyset''\equiv_T \text{\sc Tot}\leq_T \emptyset'$, which is a contradiction. Therefore, $A$ is not $\Sigma^0_2$.
\end{proof}

\begin{cor}
\label{cor:S3}
$\REC \notmodels \Gdelta$. 
\end{cor}
\begin{proof}
Consider the computable $G_{\delta}$ set $G$ from Proposition~\ref{prop:gdelta}. Note that $\mu(G)$ is irrational, or else $\mu(G)\geq q$ would clearly be $\Sigma^0_2$. Suppose that there is a computable $F_{\sigma}$ set $F$ such that $\mu(G) = \mu(F)$, so $\mu(G)\geq q$ if and only if $\mu(F)\geq q$. (Here, we do not even need to assume that $F\subseteq G$.) Recall that $\mu(F)\geq q$ if and only if $\mu(F^c)\leq 1-q$. Since $\mu(G)$ is irrational, $1-\mu(G)$ is irrational, so $\mu(F^c) \leq 1-q$ is a $\Sigma^0_2$ predicate. But $\mu(F^c) \leq 1-q$ is equivalent to $\mu(G)\geq q$, which is a contradiction.  
\end{proof}

The following proposition just says that there are $\Sigma^0_1$ classes in $2^\omega$ with arbitrarily small measure that contain all computable sets. This is well known: consider the $\Sigma^0_1$ classes that make up a universal Martin-L\"of test $\{U_n\}_{n\in\omega}$.

\begin{prop}
\label{prop:closed}
Let $\epsilon > 0$. There is a computable closed set $C$ such that $C$ contains no computable elements and  $\mu(C)\geq 1-\epsilon$.
\end{prop}
\begin{proof}
We define a computable open set set $O$ such that $O$ contains all of the computable sets and $\mu(O)\leq\epsilon$. Fix $n\in\omega$ such that $2^{-n}\leq\epsilon$. We enumerate $O$ in stages. At stage $s$, we check for every $e\leq s$ if $\phi_e(x)$ has converged and taken values in $\{ 0,1 \}$ for all $x\leq n+e$. For those $e$ for which this happens, we enumerate $\langle \varphi_e(0), \ldots, \varphi_e(n+e) \rangle$ into $O_s$.  

It is clear that $O$ will contain all of the computable sets. Furthermore, each $e\in\omega$ adds at most $2^{-(n+e+1)}$ to the measure of $O$.  Therefore, $\mu(O)\leq\sum_{e=0}^{\infty} 2^{-(n+e+1)}=2^{-n}\leq\epsilon$.
\end{proof}

\begin{cor}
\label{cor:pos}
$\REC \models \Positive$ and $\REC \models \Gepsilon$.
\end{cor}
\begin{proof}
To see that $\REC \models \Positive$, fix any computable $G_{\delta}$ set $G$ such that $\mu(G) > 0$.  
By Proposition~\ref{prop:closed}, there is a computable closed set $C$ such that $\mu(C) > 0$ and 
$C$ contains no computable elements.  Therefore, $C$ is a code for a closed set in the 
$\omega$-model $\REC$ and $\REC \models C = \emptyset$ (in the sense that $\REC \models \neg \exists X ( X \in C)$), hence 
$\REC \models C \subseteq G$.

Since $C$ is a computable closed set, we can fix a computable prefix free code 
$O$ for the complement of $C$.  Because $\mu(C) > 0$, there is a rational $q < 1$ such that $\forall t \, (\mu(O_t) \leq q)$.  
Since $\mu(O_t) \leq q$ is an arithmetic fact and $\REC$ is an $\omega$-model, $\REC \models \forall t \, (\mu(O_t) \leq q)$ 
and hence $\REC \models \mu(C) > 0$.  Therefore, $\REC \models \Positive$.  

The proof that $\REC \models \Gepsilon$ is the same except that we start with $C$ such that 
$\mu(C) \geq \mu(G) - \epsilon$ for the given $\epsilon$.
\end{proof}

\begin{cor}
$\RCA\nvdash \Positive \rightarrow \Gdelta$ and $\RCA\nvdash \Gepsilon \rightarrow \Gdelta$.
\end{cor}
\begin{proof}
This corollary follows immediately from Corollaries~\ref{cor:S3} and \ref{cor:pos}.
\end{proof}

\section{Logarithm Properties}
\label{sec:log}

We have now established that although positive measure domination is equivalent 
to uniform almost everywhere domination, $\RCA$ is not strong enough to 
prove $\Positive \rightarrow \Gdelta$.  In the last two sections, we show that $\WWKL$ is strong enough to prove this implication.  
In this section, we sketch the development of the natural logarithm in $\RCA$ and give an analogue of Lemma \ref{lem:analysis}.  

We wish to define the natural logarithm using the usual integral form
\[
\ln(x) = \int_{1}^{x} \frac{1}{u} \, du.
\]
Because the function $f(u) = 1/u$ does not have a modulus of uniform continuity, we do not automatically obtain
a code for $\ln(x)$ as a continuous function in \RCA.  
(See Simpson \cite{si:09}, Definition IV.2.1, Lemma IV.2.6, and Theorem IV.2.7 for the 
relevant background on integrals in subsystems of second order arithmetic.)  

Let $q \in \mathbb{Q}^{+}$.  Following the standard procedure for estimating $\int_{1}^{q} \frac{1}{u} \, du$ by rectangles, 
we subdivide the interval $[1,q]$ (or $[q,1]$ if $q<1$) into $n$ equal pieces.  Because $f(u) = 1/u$ is a decreasing function, we obtain 
upper and lower estimates of the integral using the left and right endpoints of each interval to define the height of the approximating rectangle.    
A short calculation shows that
\[
\text{Upper Sum} \, - \, \text{Lower Sum} \, = \, \frac{|q-1|}{n}\left|1-\frac{1}{q}\right|,
\]
which goes to $0$ as $n \rightarrow \infty$.

In $\RCA$, we define the following code for $\ln(x)$.  (See Simpson \cite{si:09}, Definition II.6.1, for the formal 
definition of a code for a continuous function in a subsystem of second order arithmetic.)  Let  
\[
\Phi_{ln} \subseteq \mathbb{N} \times \mathbb{Q}^{+} \times \mathbb{Q}^{+} \times \mathbb{Q}^{+} 
\times \mathbb{Q}^{+}
\]
be given by $(n,a,r,b,s) \in \Phi_{ln}$ if and only if $0 < a-r$, the upper sum for the estimate of 
$\ln(a+r)$ using $n$ intervals is $< b+s$, and the lower sum for the estimate of $\ln(a-r)$ using 
$n$ intervals is $> b-s$.  Since the difference between the upper and lower sums converges to $0$, 
$\Phi_{ln}$ is a code for a continuous function and the function $\ln(x)$ defined by 
these conditions coincides with $\int_{1}^{x} 1/u \, du$.  The proof that $1/x$ is the derivative of 
$\ln(x)$ can be carried out in a straightforward manner within $\RCA$.

\begin{lem}[\RCA] 
\label{lem:logprod}
The following results hold.  
\begin{enumerate}
\item[(1)] The Mean Value Theorem.
\item[(2)] If $f$ is a differentiable function on an open interval in $\mathbb{R}$, then $f' = 0$ on this 
interval if and only if $f$ is constant.  If $f' \geq 0$ on this interval, then $f$ is nondecreasing, and if $f' \leq 0$ on this interval, 
then $f$ is nonincreasing.  
\item[(3)] For all $a,b \in \mathbb{R}^{+}$, $\ln(ab) = \ln(a) + \ln(b)$.
\item[(4)] For all $k \in \mathbb{N}$ and all sequences of positive rational numbers $a_0, \ldots, a_k$, 
$\ln(\prod_{i=0}^{k} a_i) = \sum_{i=0}^k \ln(a_i)$.    
\end{enumerate}
\end{lem} 

\begin{proof}
Part (1) is proved by Hardin and Velleman in \cite{har:01}.  Parts (2) and (3) follow by their 
classical proofs using the Mean Value Theorem.  Part (4) follows by $\Pi^0_1$ induction on 
$k$ since the equality predicate between reals is $\Pi^0_1$.  
\end{proof}

\begin{lem}[\RCA]
\label{lem:lower}
For $0 \leq x < 1$, $x \leq | \ln(1-x)|$.
\end{lem}

\begin{proof}
Consider the function $f(x) = -x - \ln(1-x)$.  Since $f(0) = 0$ and 
\[
f'(x) = -1+\frac{1}{1-x} \geq 0
\]
for $0 \leq x < 1$, $f(x)$ is nondecreasing and nonnegative on $[0,1)$. But $-x-\ln(1-x) \geq 0$ implies that $x \leq | \ln(1-x)|$.
\end{proof}

\begin{lem}[\RCA] 
\label{lem:upper}
For $0 \leq x\leq 1/2$, $| \ln(1-x)| \leq 2x$.
\end{lem}

\begin{proof}
Consider the function $f(x) = -2x - \ln(1-x)$.  Since $f(0) = 0$ and 
\[
f'(x) = -2 + \frac{1}{1-x} \leq 0
\]
for $0 \leq x \leq 1/2$, $f(x)$ is nonincreasing and nonpositive on $[0,1/2]$. But $-2x-\ln(1-x) \leq 0$ implies that
$|\ln(1-x)| \leq 2x$.
\end{proof}

\begin{defn}[\RCA] 
Let $a_i$, $i \in \mathbb{N}$, be a sequence of real numbers.  
$\sum_{i=0}^{\infty} a_i$ is \emph{bounded above} if there is a rational $q$ such that 
for every $k$, $\sum_{i=0}^k a_i \leq q$.  (We do not assume that the infinite series converges for 
this definition.)  Similarly, $\sum_{i=0}^{\infty} a_i$ is \emph{bounded below} if there 
is a rational $q$ such that for every $k$, $\sum_{i=0}^k a_i \geq q$.
\end{defn}

\begin{defn}[\RCA] 
Let $b_i$, $i \in \mathbb{N}$, be a sequence of real numbers such that $0 < b_i \leq 1$.  $\prod_{i=0}^{\infty} b_i$ is 
\emph{bounded away from $0$} if there is a rational $q > 0$ such that for every $k$, 
$\prod_{i=0}^k b_i \geq q$.  
\end{defn}

Finally, we arrive at the version of Lemma~\ref{lem:analysis} that we will use in the next section.  

\begin{prop}[\RCA] 
\label{prop:converge}
Let $\langle a_i \mid i \in \mathbb{N} \rangle$ be a sequence of rational numbers such that 
$0 \leq a_i < 1$.  $\sum_{i=0}^{\infty} a_i$ is bounded above if and only if 
$\prod_{i=0}^{\infty} (1-a_i)$ is bounded away from $0$.
\end{prop}
\begin{proof}
For both expressions, the only way they can be bounded as desired is if $a_i$ converges to $0$, in particular for all but finitely many $i$ we have $0\le a_i\le 1/2$. So by Lemmas \ref{lem:lower} and \ref{lem:upper}, $\sum_{i=0}^{\infty} a_i$ is bounded above if and 
only if $\sum_{i=0}^{\infty} |\ln(1-a_i)|$ is bounded above.   
Because $\ln(1-a_i) = -|\ln(1-a_i)|$, $\sum_{i=0}^{\infty} \ln(1-a_i)$ is bounded below 
if and only if $\sum_{i=0}^{\infty} |\ln(1-a_i)|$ is bounded above.  Therefore, to finish the 
proof, it suffices to show that $\sum_{i=0}^{\infty} \ln(1-a_i)$ is bounded below if and only if 
$\prod_{i=0}^{\infty} (1-a_i)$ is bounded away from $0$.  By Part (3) of Lemma \ref{lem:logprod}
\[
\sum_{i=0}^k \ln(1-a_i) \geq q \Leftrightarrow \ln \left( \prod_{i=0}^k (1-a_i) \right) \geq q \Leftrightarrow 
\prod_{i=0}^k (1-a_i) \geq e^q > 0.
\]
(We omit the straightforward details of developing the exponential function as the inverse of the natural log.)  
\end{proof}

We will also want a more explicit version of one direction of Lemma~\ref{lem:analysis}.

\begin{prop}[\RCA]
\label{prop:product}
Let $k\in\mathbb{N}$ and let $\langle a_i \mid 0\le i\le k \rangle$ be a sequence of rational numbers such that 
$0\leq a_i\leq \frac{1}{2}$. If $\sum_{i=0}^{k} a_i \leq 2$, then $\prod_{i=0}^{k} (1-a_i) \geq \frac{1}{81}$.
\end{prop}
\begin{proof}
If $0\leq a_i\leq \frac{1}{2}$, then by Lemma~\ref{lem:upper}, $0\leq -\ln(1-a_i)\leq 2a_i$. Thus 
\[
\sum_{i=0}^k \ln(1-a_i)\ge \sum_{i=0}^k (-2a_i) =(-2)\sum_{i=0}^k a_i \ge -4
\]
so as in Proposition \ref{prop:converge}, $\prod_{i=0}^k(1-a_i)\ge e^{-4}\geq 1/81$ (using the fact that $e\leq3$).
\end{proof}

\section{Working in \WWKL}
\label{sec:wwkl}

Throughout this section, we work in $\WWKL$ to prove $\Positive \rightarrow \Gdelta$.  
Our proof will roughly be a formalization of the arguments in Lemma \ref{lem:measure} and Corollary \ref{cor:pmd} with one important difference. In the proofs leading to Corollary \ref{cor:pmd}, we used the fact that every $\Pi^0_2$ 
class contains a $\Sigma^{\emptyset'}_2$ class of the same measure.  This fact allowed us 
to switch from working with a $\Pi^0_2$ class to working with closed classes with oracles.  
Because $\WWKL$ cannot prove the existence of $\emptyset'$, we need to work directly with the given $G_{\delta}$ set 
and approximate its measure within $\WWKL$.   Throughout this 
section we work in $\WWKL$ (in fact, except for Lemma \ref{lem:sigma01}, we work in $\RCA$), assume $\Positive$ and prove $\Gdelta$.  

Let $X = \langle X_i \mid i \in \mathbb{N} \rangle$ be a code for a $G_{\delta}$ set of positive measure.  Each 
$X_i$ is a nonempty prefix-free subset of $2^{< \mathbb{N}}$ and $X_{i,s}$ denotes the set of all strings $\tau \in X_i$ 
such that $|\tau| \leq s$.  We will be notationally sloppy about the distinction between coding sets, such as $X$ and $X_i$, and the 
subsets of $2^{\mathbb{N}}$ they code, relying on the context to indicate which is the intended meaning.  If the context is 
not clear, we will use square brackets $[X]$ to denote the coded subset of $2^{\mathbb{N}}$.     

For each pair $i, n \in \mathbb{N}$, we define a function $m_{i,n}(t)$ by primitive recursion 
(uniformly in $i$ and $n$) to approximate $\mu(X_i)$.  Set $m_{i,n}(0) = 0$ and 
\[
m_{i,n}(t+1) = 
\left\{ 
\begin{array}{ll} 
m_{i,n}(t) & \text{if} \, \mu(X_{i,t+1} - X_{i,m_{i,n}(t)}) < 2^{-n-i-1}, \\
t+1 & \text{otherwise.}
\end{array}
\right.
\]

\begin{lem}
The following properties hold for each $i,n \in \mathbb{N}$.  
\begin{enumerate}
\item[(1)] $\forall t,u \, (t < u \rightarrow m_{i,n}(t) \leq m_{i,n}(u))$.
\item[(2)] $\forall t,u \, (m_{i,n}(t) < m_{i,n}(u) \rightarrow (t<u \wedge \mu(X_{i,m_{i,n}(u)} - 
X_{i,m_{i,n}(t)}) \geq 2^{-i-n-1}))$.
\item[(3)] $\exists t \forall u \geq t \, (m_{i,n}(u) = m_{i,n}(t))$.
\end{enumerate}
\end{lem}

\begin{proof}
Properties (1) and (2) follow directly from the definitions.  To prove Property (3), we proceed by 
contradiction.  If Property (3) fails for a particular $i$ and $n$, then by Property (1), for all $t$, 
there is a $u > t$ such that $m_{i,n}(u) > m_{i,n}(t)$.  We define a function $f$ such that 
$f(0) = 0$ and $f(j+1) =$ the least $u > f(j)$ such that $m_{i,n}(u) > m_{i,n}(f(j))$.  By Property (2), 
we have that $\mu(X_{i,m_{i,n}(f(j))}) \geq j \cdot 2^{-i-n-1}$, which for $j > 2^{i+n+1}$ gives the 
desired contradiction.  
\end{proof}

We let $m_{i,n}^{\infty} = \lim_{s} m_{i,n}(s)$. (So in a sense $m_{i,n}^\infty$ is the last stage that is significant for the pair $(i,n)$.)   As we are working in \WWKL, we cannot form a function taking each pair 
$\langle i,n \rangle$ to $m_{i,n}^{\infty}$, so we understand each 
statement $m_{i,n}^{\infty} = k$ to be an abbreviation for the $\Delta^0_2$ formula 
given by the equivalent formulations $\exists t \forall u \geq t (m_{i,n}(u) = k)$ and 
$\forall t \exists u \geq t (m_{i,n}(u) = k)$.  

We say that $\langle \sigma,n \rangle \in \mathbb{N}^{< \mathbb{N}} \times \mathbb{N}$ is 
\emph{correct at} $s$ if $|\sigma|\le s$, $n\le s$, and $\sigma(i) = m_{i,n}(s)$ for all $i < |\sigma|$.  (The collection of triples 
$\langle \sigma, n, s \rangle$ such that $\langle \sigma,n \rangle$ is correct at $s$ is a set.)  We say that 
$\langle \sigma,n \rangle$ is \emph{correct} if $\sigma(i) = m_{i,n}^{\infty}$ for all $i < |\sigma|$ and 
we let $\mathbf{C_n^{\infty}}$ denote the $\Delta^0_2$ 
class of all strings $\sigma$ such that $\langle \sigma,n \rangle$ is correct.  (To help maintain the distinction between 
sets of strings and classes of strings, we use boldface letters for classes.  Any statement involving a class is to be 
regarded as shorthand for the statement given by substituting in the defining formula for the class.)  Notice that in addition to 
being a $\Delta^0_2$ class, $\mathbf{C_n^{\infty}}$ is also d.c.e.\ (a difference of two computably enumerable sets) in the sense that if $\langle \sigma,n \rangle$ becomes correct 
at $s$, then either $\langle \sigma,n \rangle$ remains correct at all future stages (and $\sigma \in \mathbf{C_n^{\infty}}$) or 
$\langle \sigma,n \rangle$ ceases to be correct at some $t > s$ and is never correct at any stage $\geq t$.  

We need to define the appropriate version of the set $I$ from Lemma \ref{lem:measure} 
for our argument.  Consider an arbitrary $n$, a stage $s$, and a value $k \leq s$.  The string 
$\sigma = \langle m_{0,n}(s), m_{1,n}(s), \ldots, m_{k-1,n}(s) \rangle$ is the unique string of length $k$ such that 
$\langle \sigma,n \rangle$ is correct at $s$.  It gives rise to the following sequence of clopen sets
\[
(X_{0,\sigma(0)})^c \subseteq (X_{0,\sigma(0)} \cap X_{1,\sigma(1)})^c  
\subseteq \cdots \subseteq \left(\bigcap_{j < |\sigma|} X_{j,\sigma(j)}\right)^c.
\]
The difference $(\bigcap_{j < |\sigma|} X_{j,\sigma(j)})^c - 
(\bigcap_{j < |\sigma|-1} X_{j,\sigma(j)})^c$ is a clopen set generated by a finite set of 
minimal length strings (so these strings form an antichain).  
We define the set $I \subseteq \mathbb{N}^{< \mathbb{N}} \times 2^{< \mathbb{N}} \times 
\mathbb{N} \times \mathbb{N}$ by $\langle \sigma, \tau, n, s \rangle \in I$ if and only if 
$\langle \sigma,n \rangle$ is correct at $s$ and $\tau$ is a minimum length string used to cover 
$(\bigcap_{j < |\sigma|} X_{j,\sigma(j)})^c - (\bigcap_{j < |\sigma|-1} X_{j,\sigma(j)})^c$.

We will be interested in the following projections and restrictions of $I$.  
\begin{gather*}
I_{\sigma,n,s} = \{ \tau \mid \langle \sigma,\tau,n,s \rangle \in I \} \\
I_s = \{ \langle \sigma, \tau, n \rangle \mid \langle \sigma, \tau, n, s \rangle \in I \} \\
I_{n,s}^{\exists \sigma} = \{ \tau \mid \exists \sigma ( \langle \sigma, \tau, n, s \rangle \in I) \} \\
\mathbf{I^{\infty}} = \{ \langle \sigma, \tau, n \rangle \mid \exists t \forall s \geq t 
( \langle \sigma, \tau, n, s \rangle \in I \} \\
\mathbf{I_{\sigma,n}^{\infty}} = \{ \tau \mid \exists s (\langle \sigma, \tau, n, s \rangle \in I \}
\end{gather*}
$I_{\sigma,n,s}$, $I_s$ and $I_{n,s}^{\exists \sigma}$ are all finite sets, while 
$\mathbf{I^{\infty}}$ is a $\Delta^0_2$ class of strings (via the equivalent condition $\forall t \exists s \geq t (\langle \sigma, \tau, n, s 
\rangle \in I)$) and $\mathbf{I_{\sigma,n}^{\infty}}$ is a $\Sigma^0_1$ class of strings.  (To see 
that $I_{n,s}^{\exists \sigma}$ is a finite set, notice that $I_{n,s}^{\exists \sigma}$ is the union of the finite sets $ I_{\mu,n,s}$ over the finitely many $\mu$ such that $\langle \mu,n \rangle$ is correct 
at $s$.)  
The following properties are easily verified from the definitions.  In the current argument, Property (7) plays the 
role of the Kraft inequality in Lemma \ref{lem:measure}.

\begin{lem}
\label{lem:basicprops}
The following properties hold for all $\sigma$, $\tau$, $n$ and $s$.
\begin{enumerate}
\item[(1)] If $\langle \sigma,n \rangle$ is not correct at $s$, then $I_{\sigma,n,s} = \emptyset$.
\item[(2)] If $\langle \sigma, n \rangle$ is correct at $s$, then $I_{\sigma,n,s} \subseteq I_{n,s}^{\exists \sigma}$.
\item[(3)] $\langle \sigma, \tau, n \rangle \in \mathbf{I^{\infty}}$ if and only if $\langle \sigma,n \rangle$ is correct and 
$\tau \in \mathbf{I_{\sigma,n}^{\infty}}$.  Furthermore, if $\langle \sigma,n \rangle$ is correct and is correct at $s$, then 
$\mathbf{I_{\sigma,n}^{\infty}} = I_{\sigma,n,s}$.  
\item[(4)] For each $n$ and $k$, there is a unique string $\sigma$ such that $|\sigma| = k$ and $\langle \sigma,n \rangle$ 
is correct (that is, $\sigma \in \mathbf{C_n^{\infty}}$).  For each $i < k$, $\langle \sigma \upharpoonright i,n \rangle$ is correct, 
\[
\left(\bigcap_{i<k} X_i\right)^c \subseteq \left(\bigcap_{i < k} X_{i,\sigma(i)}\right)^c = 
\left(\bigcap_{i < k} X_{i,m_{i,n}^{\infty}}\right)^c = 
\bigcup_{i < k} [\mathbf{I_{\sigma \upharpoonright i,n}^{\infty}}]
\]
and 
\[
\mu\left(\bigcup_{i < k} [\mathbf{I_{\sigma \upharpoonright i,n}^{\infty}}] - 
\left(\bigcap_{i<k} X_i\right)^c\right) \leq \sum_{i<k} 2^{-n-i-1}.
\]
\item[(5)] Extending Property (3), for each fixed $n$, 
\[
\mu\left(\bigcup_{\sigma \in \mathbf{C_n^{\infty}}} [\mathbf{I^\infty_{\sigma,n}}] - X^c\right) = 
\mu\left(\bigcup_{\sigma \in \mathbf{C_n^{\infty}}} [\mathbf{I^\infty_{\sigma,n}}] - 
\left(\bigcap_{i \in \mathbb{N}} X_i\right)^c\right) \leq \sum_{i=0}^{\infty} 2^{-n-i-1} = 2^{-n}.
\]
\item[(6)] $I_{\sigma,n,s}$ and $I_{n,s}^{\exists \sigma}$ are finite antichains and therefore 
\[
\sum_{\tau \in I_{\sigma,n,s}} 2^{-|\tau|-n} \leq \sum_{\tau \in I_{n,s}^{\exists \sigma}} 2^{-|\tau|-n} 
\leq 2^{-n} \cdot \sum_{\tau \in I_{n,s}^{\exists \sigma}} 2^{-|\tau|} \leq 2^{-n}.
\]
\item[(7)] For any fixed $s$, 
\[
\sum_{n \in \mathbb{N}} \sum_{\tau \in I_{n,s}^{\exists \sigma}} 2^{-|\tau|-n} \leq \sum_{n \in \mathbb{N}} 
2^{-n} \leq 2
\]
and therefore 
\[
\sum_{\langle \sigma, \tau, n \rangle \in I_s} 2^{-|\tau| -n} \leq 2.
\]
\end{enumerate}
\end{lem}

Using these ideas, we define the following $\Pi^0_3$ class $\mathbf{Z}$.  (We use boldface type for $\mathbf{Z}$ since it 
is introduced via a formula rather than a set code.)  
\[
\mathbf{Z} = \bigcap_{n \in \mathbb{N}} 
\mathop{\bigcup_{s \in \mathbb{N}}}_{\sigma \in \mathbb{N}^{< \mathbb{N}}}
\bigcap_{t \geq s} [I_{\sigma,n,t}] = 
\bigcap_{n \in \mathbb{N}} \bigcup_{\sigma \in \mathbf{C_n^{\infty}}} [\mathbf{I_{\sigma,n}^{\infty}}]
\]
To be clear, since this definition involves a class predicate, it is to be read in terms of the defining 
formulas.  That is 
\begin{gather*}
A \in \mathbf{Z} \, \Leftrightarrow \, \forall n \exists \sigma, s \forall t \geq s \exists \tau \in 
I_{\sigma,n,t} (A \in [\tau]) \\
\Leftrightarrow \, \forall n \exists \sigma ( \langle \sigma,n \rangle \, \text{is correct} \, \wedge 
\exists \tau \in \mathbf{I_{\sigma,n}^{\infty}} (A \in [\tau])).
\end{gather*}
Since $\exists \tau \in I_{\sigma,n,t}$ is a bounded quantifier, $\langle \sigma,n \rangle$ is correct is a 
$\Sigma^0_2$ statement, and $\exists \tau \in \mathbf{I_{\sigma,n}^{\infty}}$ is a $\Sigma^0_1$ 
statement, each of these equivalent definitions is $\Pi^0_3$.    

\begin{lem}
$\mathbf{Z}$ has the following properties.
\begin{enumerate}
\item[(1)] $X^c \subseteq \mathbf{Z}$.
\item[(2)] $\mu(\mathbf{Z} - X^c) = 0$.
\end{enumerate} 
\end{lem}

\begin{proof}
To establish (1), for any fixed $n \in \mathbb{N}$, we have 
\[
X^c = \left(\bigcap_{i \in \mathbb{N}} X_i\right)^c \subseteq 
\left(\bigcap_{i \in \mathbb{N}} X_{i,m_{i,n}^{\infty}}\right)^c = \bigcup_{\sigma \in \mathbf{C_n^{\infty}}} 
[\mathbf{I_{\sigma,n}^{\infty}}]
\]
and therefore
\[
X^c \subseteq \bigcap_{n \in \mathbb{N}} \bigcup_{\sigma \in \mathbf{C_n^{\infty}}} 
[\mathbf{I_{\sigma,n}^{\infty}}] = \mathbf{Z}.
\]

To establish (2), for any fixed $n \in \mathbb{N}$, we have by Property (5) of Lemma \ref{lem:basicprops},  
\[
\mu\left( \bigcup_{\sigma \in \mathbf{C_n^{\infty}}} [\mathbf{I_{\sigma,n}^{\infty}}] - X^c\right) \leq 2^{-n}
\]
and therefore 
\[
\mu\left(\bigcap_{n \in \mathbb{N}} \bigcup_{\sigma \in \mathbf{C_n^{\infty}}} 
[\mathbf{I_{\sigma,n}^{\infty}}] - X^c\right) = 0.\qedhere
\]
\end{proof}

Now that we have a nicely approximated $\Pi^0_3$ superset 
$\mathbf{Z}$ of $X^c$ such that $\mu(\mathbf{Z}) = \mu(X^c)$, it remains to find a $\Pi^0_2$ superset 
$\mathbf{Y}$ of $\mathbf{Z}$ such that $\mu(\mathbf{Y}) = \mu(\mathbf{Z})$.  $\mathbf{Y}^c$ will be our desired 
$F_{\sigma}$ subset of $X$ of the same measure.

Fix a bijection between $\mathbb{N}$ and $\mathbb{N}^{< \mathbb{N}} \times 2^{< \mathbb{N}} \times \mathbb{N}$ 
and let $\langle \sigma_j, \tau_j, n_j \rangle$ denote the triple coded by $j$. Let $V_s$, $s\in\mathbb{N}$, be as in Lemma \ref{2009} for the function $f(\langle \sigma_j, \tau_j, n_j \rangle)=|\tau_j|+n_j$ and note that $V_s$, $s\in\mathbb{N}$, are defined by primitive recursion on $j$.  By Lemma \ref{2009}, for each $s$, 
$\mu([V_s]) = 2^{-|\tau_s|-n_s}$ and $\mu([V_s]^c) = 1-2^{-|\tau_s|-n_s}$.  Furthermore, because the $V_s$ sets are independent, 
if $K \subseteq \mathbb{N}$ is finite, then $\mu(\bigcap_{s \in K} [V_s]^c) = \prod_{s \in K} (1-2^{-|\tau_s| - n_s})$.

Next, we define the $G_{\delta}$ set (i.e., a $\Pi^0_2$ class) $P = \bigcap_{i \in \mathbb{N}} P_i$.  Fix a bijection between $\mathbb{N}$ and 
$\mathbb{N}^{< \mathbb{N}} \times 2^{< \mathbb{N}} \times \mathbb{N} \times \mathbb{N}$.  Let $\langle \sigma_i, \tau_i, n_i, s_i \rangle$ 
denote the tuple coded by $i$.  Define $P_i \subseteq 2^{< \mathbb{N}}$ as a union $P_i = \bigcup_{s \geq s_i} P_{i,s}$ of nested finite sets 
of strings as follows.  
If $\langle \sigma_i, \tau_i, n_i \rangle \notin I_{s_i}$, then $P_{i,s} = \{ \lambda \}$ for all $s \geq s_i$, where $\lambda$ denotes the empty string. So $P_i = \{ \lambda \}$ and 
$[P_i] = 2^{\mathbb{N}}$.  If $\langle \sigma_i, \tau_i, n_i \rangle \in I_{s_i}$, then set $P_{i,s_i} = $ a finite set of strings so that 
$[P_{i,s_i}] = [V_{\sigma_i, \tau_i, n_i}]^c$.  
For $t > s_i$, check to see if $\langle \sigma_i, \tau_i, n_i \rangle \in I_t$.  If so, then 
$P_{i,t} = P_{i,t-1} = P_{i,s_i}$.  If not, then at the first $t > s_i$ at which $\langle \sigma_i, \tau_i, n_i \rangle \notin I_t$, we extend $P_{i,t}$ 
(using strings of length $> t$) to a finite set of strings such that $[P_{i,t}] = 2^{\mathbb{N}}$, and for all  
$u > t$, we set $P_{i,u} = P_{i,t} = P_i$.  Note that for each $i$, 
either $P_{i,s} = P_{i,t}$ for all $s,t \geq s_i$ or there is a unique $t > s_i$ such that $P_{i,t} \neq P_{i,t-1}$.
 
\begin{lem}
$\forall j \exists u \forall i \leq j (P_{i,u} = P_i)$.
\end{lem}

\begin{proof}
Suppose the lemma is false and fix $j$ such that for all stages $u$, there is an 
$i \leq j$ such that $P_{i,u} \neq P_i$.  In other words, for all $u$, there is an $i \leq j$ and a stage $t > u$ such that $P_{i,t} \neq P_{i,t-1}$.   
Let $m = \max \{ s_i \mid i \leq j \}$.  Define a one-to-one function $f: \mathbb{N} \rightarrow \mathbb{N}$ by $f(0) =$ the least $t$ such that 
$t > m$ and $\exists i \leq j (P_{i,t} \neq P_{i,t-1})$ and $f(n+1) =$ the least $t$ such that $t > f(n)$ and $\exists i \leq j (P_{i,t} \neq P_{i,t-1})$.  
By Bounded $\Sigma^0_1$ Comprehension, let $A = \{ t \mid \exists n \leq j+1 (f(n) = t) \}$.  Since $|A| = j+2$, there must be a value 
$i \leq j$ and stages $t_1, t_2 \in A$ with $t_1 \neq t_2$, $P_{i,t_1} \neq P_{i,t_1-1}$ and $P_{i,t_2} \neq P_{i,t_2-1}$.  These stages 
$t_1,t_2$ contradict the fact that there is at most one stage $t > s_i$ for which $P_{i,t} \neq P_{i,t-1}$, completing the proof of this lemma.  
(Note that despite this proof, we cannot assume the existence of a function $g$ such that for all $i$, $P_{i,g(i)} = P_i$.)
\end{proof}

\begin{lem}
\label{lem:Pform}
$\displaystyle P = \bigcap_{\langle \sigma, \tau, n \rangle \in \mathbf{I^{\infty}}} [V_{\sigma, \tau, n}]^c$.
\end{lem}

\begin{proof}
This lemma follows from two calculations.  
Consider a triple $\langle \sigma, \tau, n \rangle \in \mathbf{I^{\infty}}$.  By Property (3) of Lemma \ref{lem:basicprops}, $\langle 
\sigma,n \rangle$ is correct and $\tau \in \mathbf{I_{\sigma,n}^{\infty}}$.  Fix the least $s$ such that $\langle \sigma,n \rangle$ is correct at $s$, 
and hence $\langle \sigma,n \rangle$ is correct at every $t \geq s$.  Because $s$ is chosen least, for all $u < s$, $\langle \sigma, n \rangle$ is not 
correct at $u$ and hence for all $i$ of the form $\langle \sigma, \tau, n, u \rangle$ for $u < s$, we have $[P_i] = 2^{\mathbb{N}}$.  
On the other hand, because $\tau \in \mathbf{I_{\sigma,n}^{\infty}}$, $\langle 
\sigma, \tau, n \rangle \in I_t$ for all $t \geq s$.  Therefore, for all $i$ of the form $\langle \sigma, \tau, n,t \rangle$ for $t \geq s$, 
we have $[P_i] = [V_{\sigma, \tau, n}]^c$.  

Consider a triple $\langle \sigma, \tau, n \rangle \notin \mathbf{I^{\infty}}$.  Fix any $i$ of the form $\langle \sigma, \tau, n,s \rangle$.  
First, suppose that $\langle \sigma, n \rangle$ is not correct.  Then there is a $t \geq s$ such that $\langle \sigma, n \rangle$ is not 
correct at $t$.  By Property (1) of Lemma \ref{lem:basicprops}, $I_{\sigma,n,t} = \emptyset$, so $\langle \sigma, \tau, n,t \rangle \notin 
I$ and $[P_i] = 2^{\mathbb{N}}$.  On the other hand, suppose that $\langle \sigma, n \rangle$ is correct and fix $t \geq s$ such that 
$\langle \sigma, n \rangle$ is correct at $t$.  By Property (3) of Lemma \ref{lem:basicprops}, $\tau \notin \mathbf{I_{\sigma,n}^{\infty}}$ 
and hence $\tau \notin I_{\sigma,n,t}$ and $\langle \sigma,\tau,n \rangle \notin I_t$.  Therefore, $[P_i] = 2^{\mathbb{N}}$.  
\end{proof}

\begin{lem}
$\mu(P) > 0$.
\end{lem}

\begin{proof}
We need to show that there is an $\epsilon \in \mathbb{Q}^{+}$ such that 
\[
\forall j \left(\mu\left( \bigcap_{i \leq j} P_i \right) \geq \epsilon\right). 
\]
We proceed by contradiction.  Suppose that for every $\epsilon > 0$, there is a $j$ such that 
$\mu(\bigcap_{i \leq j} P_i) < \epsilon$.  Fix an arbitrary $\epsilon$ and the corresponding $j$.  
Fix $u$ such that $P_{i,u} = P_i$ for all $i \leq j$.  As above, we assume $i = \langle \sigma_i, \tau_i, n_i, t_i \rangle$.  

For each $i \leq j$, $P_{i,u} = [V_{\sigma_i,\tau_i,n_i}]^c$ implies 
$\langle \sigma_i, \tau_i, n_i \rangle \in I_u \cap \mathbf{I^{\infty}}$, and $P_{i,u} \neq [V_{\sigma_i,\tau_i,n_i}]^c$ implies 
$P_{i,u} = 2^{\mathbb{N}}$.  Furthermore, because each 
$P_{i,u}$ is a finite set of strings, we can tell which of these cases applies.  Form the finite set 
\[
K = \{ \langle \sigma_i, \tau_i, n_i \rangle \mid i \leq j \wedge P_{i,u} = [V_{\sigma_i, \tau_i, n_i}]^c \} \subseteq I_u.
\]
Calculating measures, we have 
\[
\prod_{\langle \sigma_i, \tau_i, n_i \rangle \in K} (1 - 2^{-|\tau_i|-n_i})  = \mu\left(\bigcap_{i \leq j} P_{i,u}\right) = \mu\left(\bigcap_{i \leq j} P_i\right)
< \epsilon.
\]
Furthermore, we have
\[
\sum_{\langle \sigma_i, \tau_i, n_i \rangle \in K} 2^{-|\tau_i| - n_i} \leq 
\sum_{\langle \sigma,\tau,n \rangle \in I_u} 2^{-|\tau|-n} \leq 2.
\]
(The first inequality follows because $K \subseteq I_u$ and the second inequality follows from Property (7) of Lemma \ref{lem:basicprops}.)  
For a small enough value of $\epsilon$, the fact that $\prod_{\langle \sigma_i, \tau_i, n_i \rangle \in K} (1 - 2^{-|\tau_i|-n_i}) < \epsilon$ and 
$\sum_{\langle \sigma_i, \tau_i, n_i \rangle \in K} 2^{-|\tau_i| - n_i} \leq 2$ contradicts Proposition~\ref{prop:product}.  
\end{proof}

\begin{lem}
\label{lem:correct} 
For all $\sigma$, $\tau$ and $n$, $[V_{\sigma,\tau,n}] \cap P = \emptyset$ if and only if 
$\langle \sigma, n \rangle$ is correct and $\tau \in \mathbf{I_{\sigma,n}^{\infty}}$.  
\end{lem}

\begin{proof}
Suppose that $\langle \sigma,n \rangle$ is correct and $\tau \in \mathbf{I_{\sigma,n}^{\infty}}$.  By Property (3) of Lemma \ref{lem:basicprops},  
$\langle \sigma, \tau, n \rangle \in \mathbf{I^{\infty}}$.  By Lemma \ref{lem:Pform}, $[V_{\sigma,\tau,n}]^c$ is one of the intersected 
sets forming $P$ and therefore $[V_{\sigma,\tau,n}] \cap P = \emptyset$.  

Now assume that it is not the case that $\langle \sigma, n \rangle$ is correct and $\tau \in \mathbf{I_{\sigma,n}^{\infty}}$. Again by Property (3) of Lemma~\ref{lem:basicprops}, we have $\langle \sigma, \tau, n \rangle \notin \mathbf{I^{\infty}}$. So $[V_{\sigma,\tau,n}]^c$ does not occur in the intersection forming $P$. Let $s=\langle \sigma,\tau,n \rangle$. Recall how the sets $V_t$ were formed in Lemma~\ref{2009}. Let $k$ be the length of the longest string in $\bigcup_{t<s} V_t$. Consider the sequence $X = 1^k0^{f(s)}1^\N$. It follows from the construction of the sets $V_t$, $t\in\N$, that $X\in [V_s]$ but $X\in [V_t]^c$ for every $t\neq s$. Therefore, $X\in [V_{\sigma,\tau,n}] \cap P$, so $[V_{\sigma,\tau,n}] \cap P\neq \emptyset$.
\end{proof}

By Lemma \ref{lem:correct}, we can write $\mathbf{Z}$ as 
\[
A \in \mathbf{Z} \Leftrightarrow \forall n \exists \sigma, \tau ( [V_{\sigma,\tau,n}] \cap P = \emptyset 
\wedge A \in [\tau]).
\]

By $\Positive$, we can fix a closed set $Q \subseteq P$ such that $\mu(Q)>0$. Following the proof of Lemma~\ref{lem:measure}, it would make sense to define $\mathbf{J}$ to be the class containing all triples $\langle \sigma, \tau, n \rangle$ such that $[V_{\sigma,\tau,n}] \cap Q = \emptyset$. The problem is that without $\WKL$, this would not necessarily be a  $\Sigma^0_1$ condition. Since we want to work in $\WWKL$, we need a slightly different definition of $\mathbf{J}$. Take $k\in\N$ such that $\mu(Q)>2^{-k}$. Let
\[
\mathbf{J} = \{ \langle \sigma, \tau, n \rangle \mid \mu([V_{\sigma,\tau,n}] \cap Q) < 2^{-\langle \sigma, \tau, n \rangle-k-2} \}.
\]
In Section~\ref{sec:measuredef} we saw that if $O$ is an open set and $q\in\Q$, then $\mu(O)>q$ is a $\Sigma^0_1$ statement. Thus, $\mathbf{J}$ is a $\Sigma^0_1$ class.

\begin{lem}
\label{lem:sigma01}
If $[V_{\sigma,\tau,n}] \cap Q = \emptyset$, then $\langle \sigma, \tau, n \rangle\in\mathbf{J}$.
\end{lem}
\begin{proof}
This follows from $\WWKL$ and it is our only use of the principle. If $\langle \sigma, \tau, n \rangle\notin\mathbf{J}$, then $\mu([V_{\sigma,\tau,n}] \cap Q)>0$. But then $\WWKL$ implies that $[V_{\sigma,\tau,n}] \cap Q \neq \emptyset$.
\end{proof}

\begin{lem}
\label{lem:bounded}
The sum $\sum_{\langle \sigma,\tau,n \rangle \in \mathbf{J}} 2^{-|\tau|-n}$ is bounded above.
\end{lem}

\begin{proof}
Because $\mathbf{J}$ is a $\Sigma^0_1$ class, this sum can be expressed as $\sum a_i$ where the sequence 
$a_i \in \mathbb{Q}$ is determined by the enumeration of $\mathbf{J}$.  That is, $a_i = 2^{-|\tau|-n}$ if 
the $i$-th element enumerated into $\mathbf{J}$ is $\langle \sigma, \tau, n \rangle$.  (Recall that we think of a $\Sigma^0_1$ class such as $\mathbf{J}$ enumerated in stages with $\mathbf{J}_s$ equal to the finite set of tuples 
$\langle \sigma, \tau, n \rangle < s$ which are in $\mathbf{J}$ with an existential witness $< s$.)

We define an open set $R$ as follows. At the stage $s$ when $\langle \sigma, \tau, n \rangle$ goes into $\mathbf{J}$, we have $\mu([V_{\sigma,\tau,n}] \cap Q_s) < 2^{-\langle \sigma, \tau, n \rangle-k-2}$. Enumerate the clopen set $[V_{\sigma,\tau,n}] \cap Q_s$ into $R$. Note that $\mu(R)\leq \sum_{\langle \sigma, \tau, n \rangle\in\mathbf{J}} 2^{-\langle \sigma, \tau, n \rangle-k-2}\leq 2^{-k-1}$. Also note that if $\langle \sigma, \tau, n \rangle\in\mathbf{J}$, then $[V_{\sigma,\tau,n}]\subseteq R\cup Q^c$. Therefore, $Q-R\subseteq\bigcap_{\langle \sigma,\tau,n \rangle \in \mathbf{J}_s} [V_{\sigma,\tau,n}]^c$.

For any $s \in \mathbb{N}$, we have 
\begin{multline*}
\prod_{\langle \sigma,\tau,n \rangle \in \mathbf{J}_s} (1-2^{-|\tau|-n}) = \mu\left(\bigcap_{\langle \sigma,\tau,n \rangle \in \mathbf{J}_s} [V_{\sigma,\tau,n}]^c\right) \geq \mu(Q-R) \\
\geq \mu(Q) - \mu(R) > 2^{-k} - 2^{-k-1} = 2^{-k-1} > 0.
\end{multline*}
and therefore the product $\prod_{\langle \sigma, \tau, n \rangle \in \mathbf{J}} (1-2^{-|\tau|-n})$ is bounded away from $0$. Hence, by Proposition \ref{prop:converge}, $\sum_{\langle \sigma,\tau,n \rangle \in \mathbf{J}} 2^{-|\tau|-n}$ is bounded above.
\end{proof}

To approximate the defining condition for 
$\mathbf{Z}$ given immediately after Lemma \ref{lem:correct}, we look at the $\Sigma^0_1$ predicate 
\[
\langle \sigma,\tau,n \rangle \in \mathbf{J} \wedge \exists t \geq s (\langle \sigma, \tau, n \rangle \in I_t).
\]
Define 
\begin{gather*}
\mathbf{T_{n,s}} = \{ \langle \sigma, \tau, n \rangle \mid \langle \sigma,\tau,n \rangle \in \mathbf{J} \wedge 
\exists t \geq s (\langle \sigma, \tau, n \rangle \in I_t) \}, \\
\mathbf{U_{n,s}} = \{ \tau \mid \exists \sigma (\langle \sigma, \tau, n \rangle \in T_{n,s}) \}.
\end{gather*}
Note that $\mathbf{T_{n,s}}$ and $\mathbf{U_{n,s}}$ are $\Sigma^0_1$ classes and for any fixed $n$, we have 
\begin{align*}
& \mathbf{T_{n,0}} \supseteq \mathbf{T_{n,1}} \supseteq \mathbf{T_{n,2}} \supseteq \cdots, \\
\text{and } & \mathbf{U_{n,0}} \supseteq \mathbf{U_{n,1}} \supseteq \mathbf{U_{n,2}} \supseteq \cdots.
\end{align*}
We finally define our desired $\Pi^0_2$ class $\mathbf{Y}$
\[
\mathbf{Y} = \bigcap_{n \in \mathbb{N}} \bigcap_{s \in \mathbb{N}} [\mathbf{U_{n,s}}].
\]

\begin{lem}
\label{lem:contains}
$\mathbf{Z} \subseteq \mathbf{Y}$.
\end{lem}

\begin{proof}
Let $A \in \mathbf{Z}$ and fix any $n$.  We show that $A \in \bigcap_{s \in \mathbb{N}} [\mathbf{U_{n,s}}]$. 
Since $A \in \mathbf{Z}$, there are strings $\sigma$ and $\tau$ such that $[V_{\sigma,\tau,n}] 
\cap P = \emptyset$ and $A \in [\tau]$.  Since $Q \subseteq P$, we have 
$[V_{\sigma,\tau,n}] \cap Q = \emptyset$, so $\langle \sigma,\tau,n \rangle \in \mathbf{J}$. By Lemma \ref{lem:correct}, we have that 
$\langle \sigma,n \rangle$ is correct and $\tau \in \mathbf{I_{\sigma,n}^{\infty}}$.  Therefore, 
for all $s$, there is $t \geq s$ such that $\langle \sigma,\tau,n \rangle \in I_t$.  (In fact, this is 
true for almost all $t \geq s$.)  It follows that for all $s$, 
\[
\langle \sigma,\tau,n \rangle \in \mathbf{J} \wedge \exists t \geq s (\langle \sigma,\tau,n \rangle \in 
I_t),
\]
and hence that $\langle \sigma,\tau,n \rangle \in \mathbf{T_{n,s}}$ and $\tau \in \mathbf{U_{n,s}}$ for all $s$.  Since $A \in [\tau]$, we have that 
$A \in \bigcap_{s \in \mathbb{N}} [\mathbf{U_{n,s}}]$ as required. 
\end{proof}

\begin{lem}
$\mu(\mathbf{Y} - \mathbf{Z}) = 0$.
\end{lem}

\begin{proof}
For $k \in \mathbb{N}$ we let 
\begin{gather*}
\mathbf{Z_k} = \bigcup_{\sigma \in \mathbf{C_k^{\infty}}} 
[\mathbf{I_{\sigma,k}^{\infty}}], \\
\mathbf{Y_k} = \bigcap_{s \in \mathbb{N}} [\mathbf{U_{k,s}}].
\end{gather*}
The proof of Lemma \ref{lem:contains} shows that $\mathbf{Z_k} \subseteq \mathbf{Y_k}$.  Since $\mathbf{Z} = \bigcap_k \mathbf{Z_k}$ 
and $\mathbf{Y} = \bigcap_k \mathbf{Y_k}$, it suffices to show that $\mu(\mathbf{Y_k} - \mathbf{Z_k}) = 0$.  To prove this measure 
statement, we need to prove that for every $\epsilon \in \mathbb{Q}^{+}$, there is a $c$ such that 
$\mu(\mathbf{U_{k,c}} - \mathbf{Z_k}) < \epsilon$.  

Fix $k \in \mathbb{N}$ and $\epsilon \in \mathbb{Q}^{+}$.  By Lemma \ref{lem:bounded}, fix $m$ such that 
\[
\mathop{\sum_{\langle \sigma,\tau,n \rangle \in \mathbf{J}}}_{\langle \sigma,\tau,n \rangle \geq m}
2^{-|\tau|-n} < \epsilon \cdot 2^{-k}.
\]
(In this sum, $\sigma$, $\tau$ and $n$ vary.)  Fixing $n=k$ in this summation and multiplying by $2^k$, we have (now letting only 
$\sigma$ and $\tau$ vary)
\[
\mathop{\sum_{\langle \sigma,\tau,k \rangle \in \mathbf{J}}}_{\langle \sigma,\tau,k \rangle \geq m}
2^{-|\tau|} < \epsilon.
\]
For each tuple $\langle \sigma, \tau, k \rangle \in \mathbf{T_{k,0}}$ such that $\langle \sigma,\tau,k \rangle 
\notin \mathbf{I^{\infty}}$, there must be an $c$ such that for all $u \geq c$, $\langle \sigma, \tau, k \rangle \notin I_u$, and 
hence $\langle \sigma, \tau, k \rangle \notin \mathbf{T_{k,c}}$.  
We would like to obtain a single witness $c$ which works for all such $\langle \sigma,\tau,k \rangle < m$.  

Consider the bounded quantifier statement $\varphi(\sigma,\tau,k,u)$ which says that $u$ is a witness for 
$\langle \sigma,\tau,n \rangle \in \mathbf{J}$, that $\exists t \leq u ( \langle \sigma,\tau,k \rangle \in I_t)$, and that 
$\langle \sigma,\tau,k \rangle \notin I_u$.  Fix any $\langle \sigma,\tau,k \rangle$ such that 
$\exists u \, \varphi(\sigma,\tau,k,u)$, fix the witness $u$ for this statement and fix $t \leq u$ that witnesses the second 
conjunct of $\varphi$.  Because $\langle \sigma,\tau,n \rangle \in \mathbf{J}$ and $\langle \sigma, \tau, k \rangle \in I_t$, 
we have that $\langle \sigma, \tau, k \rangle \in \mathbf{T_{k,0}}$.  Because $\langle \sigma, \tau, k \rangle \notin I_u$ and $t < u$, 
we have that $\forall v \geq u (\langle \sigma, \tau, k \rangle \notin I_v)$ and hence $\langle \sigma, \tau, k \rangle \notin \mathbf{T_{k,u}}$.  Furthermore, by the previous paragraph, if $\langle \sigma, \tau, k \rangle \in \mathbf{T_{k,0}}$ and $\langle 
\sigma, \tau, k \rangle \notin \mathbf{I^{\infty}}$, then $\exists u \, \varphi(\sigma, \tau, k, u)$.  

The strong $\Sigma^0_1$ bounding scheme (which holds in $\RCA$, see Simpson \cite{si:09} Exercise II.3.14) 
implies that 
\[
\exists c \, \forall \langle \sigma, \tau, k \rangle \leq m \, ( \exists u \, \varphi(\sigma, \tau, k, u) \rightarrow \exists u \leq c \,  
\varphi(\sigma, \tau, k, u)).
\]
Fix such a $c$.  For any $\langle \sigma, \tau, k \rangle < m$, if $\langle \sigma, \tau, k \rangle \in \mathbf{T_{k,0}}$ and 
$\langle \sigma, \tau, k \rangle \notin \mathbf{I^{\infty}}$, then $\langle \sigma, \tau, k \rangle \notin \mathbf{T_{k,c}}$.  

To finish the proof, it suffices to show that 
\[
\mu(\mathbf{U_{k,c}} - \mathbf{Z_k}) \leq 
\mathop{\sum_{\langle \sigma,\tau,k \rangle \in \mathbf{J}}}_{\langle \sigma,\tau,k \rangle \geq m}
2^{-|\tau|} < \epsilon.
\]
Suppose that $\tau \in \mathbf{U_{k,c}}$ but $\tau \notin \mathbf{Z_k}$ (that is, $\tau \notin \mathbf{I_{\sigma,k}^{\infty}}$ for any 
$\sigma \in \mathbf{C_k^{\infty}}$).  Fix $\sigma$ such that $\langle \sigma, \tau, k \rangle \in \mathbf{T_{k,c}}$.  We need to show that 
$\langle \sigma, \tau, k \rangle \geq m$.  $\langle \sigma, \tau, k \rangle \in \mathbf{T_{k,c}}$ implies that 
$\exists t \geq c (\langle \sigma,\tau,k \rangle \in I_t)$ and hence 
$\tau \in \mathbf{I_{\sigma,k}^{\infty}}$.  Since $\tau \notin \mathbf{Z}$, $\langle \sigma,k \rangle$ must not 
be correct and hence $\langle \sigma, \tau, k \rangle \notin \mathbf{I^{\infty}}$ by Property (3) of Lemma \ref{lem:basicprops}.  
Suppose for a contradiction that $\langle \sigma, \tau, k \rangle < m$.  Since $\langle \sigma, \tau, k \rangle \in 
\mathbf{T_{k,c}} \subseteq \mathbf{T_{k,0}}$ and $\langle \sigma, \tau, k \rangle \notin \mathbf{I^{\infty}}$, we have 
(by our choice of $c$) that $\langle \sigma, \tau, k \rangle \notin \mathbf{T_{k,c}}$, which is the desired contradiction.  
\end{proof}

\bibliographystyle{plain}
\bibliography{domination}

\end{document}